\documentclass[preprint,12pt]{elsarticle}

\usepackage{graphicx, subfigure}
\usepackage{amsmath,amscd,amssymb,amsfonts, amsthm, epstopdf}
\usepackage[arrow, matrix, curve]{xy}
 \usepackage[hang,small,bf]{caption}
 \usepackage{hyperref}

\numberwithin{equation}{section}

\DeclareMathOperator{\im}{im} 
\DeclareMathOperator{\sgn}{sgn}
\DeclareMathOperator{\vol}{vol}

\DeclareMathOperator{\e}{e}
\DeclareMathOperator{\lk}{lk}
\DeclareMathOperator{\st}{St}
\DeclareMathOperator{\cl}{Cl}

\begin{document}
\theoremstyle{plain}
\newtheorem{thm}{Theorem}[section]
\newtheorem{lem}[thm]{Lemma}
\newtheorem{prop}[thm]{Proposition}
\newtheorem{coll}[thm]{Corollary}
\newtheorem{conj}{Conjecture}[section]

\theoremstyle{definition}
\newtheorem{defn}{Definition}[section]
\newtheorem{exmp}{Example}[section]

\theoremstyle{remark}
\newtheorem{rem}{Remark}[section]
\newtheorem*{note}{Note}

\newcommand{\rb}{\mathbb{R}}
\newcommand{\lup}{\Delta_{i}^{up}}
\newcommand{\ldown}{\Delta_{i}^{down}}
\newcommand{\s}{\mathbf{s}}
\newcommand{\dote}{\stackrel{\circ}{=}}
\newcommand{\dotu}{\stackrel{\circ}{\cup}}

\begin{frontmatter}

\title{Spectra of combinatorial Laplace operators on simplicial complexes}
\author{Danijela Horak}
\address{Max Planck Institute for Mathematics in the Sciences,
  Inselstrasse 22, D-04103 Leipzig, Germany}
\author{J\"urgen Jost}
\address{Max Planck Institute for Mathematics in the Sciences,
  Inselstrasse 22, D-04103 Leipzig, Germany\\
Department of Mathematics and Computer Science, Leipzig
  University, D-04091 Leipzig, Germany\\
Santa Fe Institute for the Sciences of Complexity, Santa Fe, NM 87501,
USA}

\begin{abstract}
 We first develop a  general framework for  Laplace operators defined in
 terms of the combinatorial structure of a simplicial complex. This
 includes, among others,  the graph Laplacian, the combinatorial
 Laplacian on simplicial complexes, the weighted Laplacian, and 
the normalized graph Laplacian.
This framework then allows us to  define  the normalized  Laplace
operator $\lup$  on simplicial complexes which we then systematically investigate.
We study the effects of 
a wedge sum, a join and a duplication of a motif on  the  spectrum of
the normalized Laplace operator, and identify 
some of the combinatorial features of a simplicial
complex that are encoded in its spectrum. 
 \end{abstract}
\begin{keyword}
normalized graph Laplacian, combinatorial Laplacian, hypergraph Laplacian, graph Laplacian, simplicial complex
\end{keyword}
\end{frontmatter}

\section{Introduction}
The  study of graph  Laplacians has a long and prolific history. It first appeared in a  paper by Kirchhoff \cite{Kirchhoff}, where he analysed electrical networks and stated  the celebrated matrix tree theorem.
The  Laplace operator $L$ of  \cite{Kirchhoff} operates on a real
valued function $f$ on the vertices of a graph as
\begin{equation}
\label{GL}
Lf(v_{i})=\deg v_{i} f(v_{i}) -\sum_{v_{i}\sim v_{j}}f(v_{j}).
\end{equation}
In spite of its rather early beginnings this topic did not gain much attention among scientists until the early 1970's and the work of Fiedler  \cite{Fiedler1973}, and his results  on  correlation among  the smallest non-zero eigenvalue  
and the connectivity of a graph.
Before Fiedler drew attention to the graph Laplacian, 
  graphs were usually characterized by  means of the spectrum of its
  adjacency matrix, but in the wake of \cite{Fiedler1973}, 
there has  been a number  of papers 
 ( e.g. \cite{GroneMerris} ) arguing  in  favour of the  graph Laplacian and its spectrum. 
 For  good survey articles on the graph Laplacian the reader is referred to 
\cite{MerrisSurvey} or \cite{MoharSurvey}.

In a different tradition,  the graph Laplacian  was generalized to simplicial complexes  by Eckmann \cite{Eckmann1944}, 
who formulated and proved  the discrete version of the Hodge theorem;
this can be formulated as 
$$
  \ker (\delta_{i}^{*}\delta_{i}+\delta_{i-1}\delta_{i-1}^{*})\cong \tilde{H}^{i}(K,\rb)
$$
where 
$$
L_{i}=\delta_{i}^{*}\delta_{i}+\delta_{i-1}\delta_{i-1}^{*}
$$
is the higher order combinatorial Laplacian. 
Many subsequent papers then studied  properties of the
 \emph{higher order combinatorial Laplacian} (see
 \cite{Dong},\cite{Friedman},\cite{DuvalMT}),  building upon
 properties of the graph Laplacian. In particular, this operator has been employed extensively in investigating the  features of networks related to  dynamics and coverings  (see  \cite{Muhammad},\cite{Alireza}). Recently   the monograph \cite{HiraniThesis} appeared, where the combinatorial Laplacian  is systematically  studied in a context of a  discrete exterior calculus.

While the graph Laplacian introduced by Kirchhoff naturally appears in
his work on electrical flows, for other processes on graphs, like
random walks or diffusion, a different operator appears. This was
first investigated 
almost a century after Kirchhoff's work by Bottema \cite{Bottema} who studied a transition probability
operator on graphs that is equivalent to the following version of 
the graph Laplace operator
\begin{equation}
\label{NGL}
\Delta f(v_{i})= f(v_{i}) -\frac{1}{\deg v_{i}}\sum_{v_{i}\sim v_{j}}f(v_{j}).
\end{equation}
It took, however, almost another one hundred  years until  a
significant advance in the study of this operator $\Delta$,
  which got to be known by the name  \emph{normalized graph Laplacian} to distinguish it from the graph Laplacian $L$ and to emphasize the fact that its eigenvalues are in the interval $[0,2]$.
In contrast to $L$, $\Delta$ is well suited for  problems  related to random walks on graphs and graph expanders.
For a good introduction to this topic the reader is invited to consult \cite{Chung} or \cite{Grigoryan}.

The main  goals of this paper 
are a systematic framework that  can be used as  a starting point for
a study of any of the above mentioned versions of the Laplace
operator, 
and the definition and investigation of    the normalized Laplacian on simplicial complexes.
 The latter is based on the simple observation that the  form of the combinatorial Laplacian is tightly connected to 
the choice of the scalar product on the coboundary vector spaces. 
On the other hand, the scalar products can be viewed in terms of  weight functions. 
Thus, by controlling the weights, we  control the   range of the eigenvalues of the Laplace operator.
Most importantly, for the normalized Laplacian, the eigenvalues are
confined to the  range $[0,i+2]$, 
where $i$ is the order of the Laplacian. This  generalizes the fact
that the eigenvalues  of the normalized graph Laplacian $\Delta$ are
in the interval $[0,2]$. 
We shall analyze  the spectrum  of this normalized Laplacian and its
connection with the combinatorial structure of the    
simplicial complex. Perhaps somewhat surprisingly, this generality
also permits us  
 to gain   new insights for its special case,  the already extensively
 studied  normalized graph Laplacian.

There have already been several attempts towards the normalization of
the combinatorial Laplace operator. In particular,  Chung in 
\cite{ChungHypergraphs}  defined a normalized Laplacian on simplicial complexes as 
$ \delta^{*}\delta+\rho\delta\delta^{*}$, where $\rho$ is a positive constant.
However,  the spectrum of this operator is not bounded by a constant.
Recently, Taszus  \cite{Taszus} suggested a normalization of the combinatorial Laplacian via its matrix form
$ D^{-1/2}L_{i}D^{-1/2} $, where $L_{i}$ is the matrix corresponding to
the operator $ \delta^{*}\delta+\delta\delta^{*}$, where the adjoints
are defined with respect to the standard scalar product and $D$ is a diagonal matrix of $L_{i}$. This operator, too, does not have bounded spectrum.
Lu and Pung in \cite{Lu} considered random walks on hypergraphs, and
to that end defined a normalized Laplacian on a uniform hypergraph $H$ as 
$\mathcal{L}_{s}(H)=\Delta(G_{H}^{s})$, where $G_{K}^{s}$ is a
$(s-1)$-dual graph (see Definition \ref{dual graph}) of a simplicial
complex (hypergraph) $H$. The drawback of this definition is that it
fails to fit into general theory and doesn't take into account higher
order relations among edges of a hypergraph.

As is already clear from Eckmann's seminal work \cite{Eckmann1944},
the Laplacians of a simplicial complex encode its basic topology, that
is, its homology groups. In terms of the spectrum, they are given by
the dimensions of the eigensets for the eigenvalue $0$. This is the same
for all the Laplace operators investigated here. These operators,
however, differ in the nonzero part of the spectrum, and thereby
encode specific combinatorial or geometric features of a (perhaps weighted) simplicial complex in
addition to its topological aspects. Many combinatorial operations
that one can perform on a simplicial complex do not affect its
homology; nevertheless, they typically leave characteristic traces in
the spectrum of a suitable Laplace operator, and that is what we are
trying to explore. In the weighted case, there is additional geometric
information that likewise influences the spectrum. Let us try to explain this aspect
from the following perspective. As is well known, from a covering of a
set, one can construct a simplicial complex, by letting an
$i$-dimensional simplex corresponds to every intersection of $i$
members of the covering. The \v{C}ech cohomology of the covering then
is isomorphic the simplicial cohomology of the resulting
complex. When, in addition, the set that is covered carries a measure,
then we can assign to every simplex in this construction a weight
equal to the measure of the corresponding intersection. We thus obtain
a weighted simplicial complex, and we can define a corresponding
Laplacian. Its spectrum then reflects the geometry of the intersection
pattern, and not only its topology. Since such intersection patterns
arise in many areas of application, for instance as colocalization
patterns of proteins in a cell \cite{} or for many geographical data
sets, we wish to propose this Laplacian spectrum as a new tool in data
analysis. This will be developed elsewhere, on the mathematical basis
of the present paper.

This paper is organized as follows.
In Section 2 we give the basic  definitions for  simplicial complexes
and recall Eckmann's discrete version of  the Hodge theorem.
We define the combinatorial Laplace operator in its full generality and provide explicit expressions.
Section 3  starts with the theorem about the number of zeros in the spectrum of the 
 the general Laplace operator.   We then discuss the effect of the scalar products on the spectrum
and obtain the upper  and lower bound on the maximum eigenvalue of the Laplacian. Finally,  we  state the definition of the \emph{normalized combinatorial Laplace operator}, which will be the main object of  the remainder of the paper.
We calculate  the spectrum of the normalized combinatorial Laplacian for some special classes of simplicial complexes  in Section 4. In particular, we discuss the spectrum of an $i$-simplex, of an orientable and a non-orientable circuit, of a path and of a star.
In Section 5 we analyse regular, pure simplicial complexes.
In Section 6 we discuss the effect of wedges, joins and duplication of motifs   on the spectrum of the normalized combinatorial Laplace operator.
In Section 7 we identify the combinatorial features of simplicial
complexes that cause the appearance of certain integer eigenvalues in
the spectrum of $\lup$. We discuss the occurrence of the eigenvalue
$i+2$ in the spectrum of $\lup$, and its connection to  the chromatic
number of the underlying graph of a complex. Furthermore, the relation among the eigenvalue $i+1$ and the duplication of vertices is established.

\section{Notations, definitions and  the combinatorial Laplace operator}

An \emph{abstract simplicial complex} $K$  on a finite set $V$ is a  collection of subsets of $V$,  which is closed under inclusion.
An $i$-face  or an $i$-simplex of $K$ is an element of cardinality $i+1$.
 $0$-faces are usually called  \emph{vertices} and $1$-faces \emph{edges}. 
 The collection of all $i$-faces of simplicial complex $K$ is denoted by $S_{i}(K)$.
The  \emph{dimension} of an $i$-face is $i$, and the 
 dimension  of  a complex   $K$ is the maximum  dimension of a    face in $K$. 
 The faces which are maximal under inclusion are called \emph{facets}. 
We say that a simplicial complex $K$ is \emph{pure} if all  facets have the same dimension.
Note that  there is a natural correspondence of hypergraphs and simplicial complexes in a natural way ( facet of a simplicial complex  corresponds to an edge in a hypergraph).
For two $(i+1)$- simplices sharing an $i$-face we use the term \emph{$i$-down neighbours}, and for two 
$i$-simplices which are faces of an $(i+1)$- simplex, we say that they are \emph{$(i+1)$-up neighbours}. 
We say that a face $F$ is \emph{oriented} if we chose an ordering on its vertices and  write  $[F]$. 
Two orderings of the vertices are said to determine \emph{the same orientation} if there is an even permutation transforming one ordering into the other. If the permutation is odd, then the orientations  are opposite.

In the remainder,  $K$ will be  an abstract simplicial complex  on a vertex set $[n]=\{1,2,\ldots, n\}$, 
 when not stated otherwise.
 The $i$-th chain group $C_{i}(K,\rb)$ of a complex $K$ with coefficients in $\rb$ is a vector space over field $\rb$ with  basis $B_{i}(K,\rb)=\{[F]\mid F \in S_{i}(K)\}$. 

The cochain groups $C^{i}(K,\rb)$ are defined as  duals of the chain groups,  i.e.
$C^{i}(K,\rb):=~\hom(K,\rb)$.  The  basis  of $C^{i}(K,\rb)$ is  given by the set of functions $\{e_{[F]}\mid [F] \in B_{i}(K,\rb))\}$ such that
$$e_{[F]}([F'])=\left\{  \begin{array}{ll}
1 & \textrm{ if } [F']=[F] \\
0 & \textrm{ otherwise. }
\end{array}
\right. $$
The functions  $e_{[F]}$ are also known as \emph{elementary cochains}.
 Traditionally, $C^{i}(K,G)$ for arbitrary group $G$, are called cochain groups. Influenced by this naming, 
we will refer to $C^{i}(K,\rb)$ as cochain \emph{groups}, although we
always keep in mind that the $C^{i}(K,\rb)$ have  
the structure of  \emph{ vector spaces}.
Note that the one-dimensional vector space  $C^{-1}(K,\rb) $  is
generated by  the identity function on the empty simplex.
We define   the  \emph{simplicial coboundary maps}
\begin{displaymath}
(\delta_{i}f)([v_{0},\ldots,v_{i+1}]) =\sum_{j=0}^{i+1}(-1)^{j}f([v_{0},\ldots,\hat{v}_{j}\ldots v_{i+1}]),
\end{displaymath}
where  $\hat{v}_{j}$ denotes that the vertex $v_{j}$ has been omitted.
The $\delta_i$ are the connecting maps in the
\emph{augmented cochain complex} of $K$ with coefficients in $\rb$,
i.e.,  the sequence of vector spaces  and linear transformations 
\begin{equation}
\label{cochain complex}
\ldots \xleftarrow{\delta_{i+1}} C^{i+1}(K,\rb) \xleftarrow{\delta_{i}}  C^{i}(K,\rb)
 \xleftarrow{\delta_{i-1}}  \ldots  \leftarrow C^{-1}(K,\rb) \leftarrow 0.
\end{equation}
Alternatively,  $\delta_{i}$ can be viewed as   the   dual of the boundary map $\partial_{i+1}$.
For a systematic treatment of  simplicial homology and cohomology the reader is referred to \cite{Hatcher}.
It is straightforward to check that   $\delta_{i}\delta_{i-1}=0$, ergo the image of $\delta_{i-1}$ is contained in the kernel of  $\delta_{i}$ and  the reduced cohomology group for every $i\geq 0$ is 
$$
\tilde{H}^{i}(K,\rb):=\ker \delta_{i}/\im \delta_{i-1}.
$$
After choosing   inner products $(\textrm{ } ,\textrm{ } )_{C^{i}}$
and $(\textrm{ } ,\textrm{ } )_{C^{i+1}}$  on  $C^{i}(K,\rb)$ and $
C^{i+1}(K,\rb)$, respectively,  the adjoint
$\delta_{i}^{*}:C^{i+1}(K,\rb)\rightarrow ~C^{i}(K,\rb)$  of the
coboundary operator $\delta_{i}$  is defined  by 
 $$
 (\delta_{i}f_{1},f_{2})_{C^{i+1}}=(f_{1},\delta_{i}^{*} f_{2})_{C^{i}},
 $$
for every $f_{1}\in C^{i}(K,\rb)$ and $f_{2}\in C^{i+1}(K,\rb)$.
\begin{defn}
We define the following three operators on $C^{i}(K,\mathbb{R})$:
\begin{enumerate}
\item[(i)] \emph{$i$-dimensional combinatorial up Laplace operator}  or simply $i$-up Laplace operator
$$
\mathcal{L}_{i}^{up}(K):=\delta_{i}^{*}\delta_{i},
$$
\item [(ii)]\emph{$i$-dimensional combinatorial down Laplace operator}  or  $i$-down Laplace operator
$$
\mathcal{L}_{i}^{down}(K):=\delta_{i-1}\delta_{i-1}^{*},
$$
\item [(iii)]\emph{$i$-dimensional combinatorial Laplace operator}  or $i$-Laplace operator
$$
\mathcal{L}_{i}(K):=  \mathcal{L}_{i}^{up}(K)+ \mathcal{L}_{i}^{down} (K)=\delta_{i}^{*}\delta_{i}+\delta_{i-1}\delta_{i-1}^{*}.
$$
\end{enumerate}
\end{defn}
Since 
\begin{equation*} 
C^{i+1}(K,\mathbb{R})
\begin{array}{l}
\underleftarrow{\delta_{i\textrm{ }\textrm{ }\textrm{ }}}\\ 
 \overrightarrow{ \delta^{*}_{i\textrm{ }\textrm{ }\textrm{ }}}
\end{array}
C^{i}(K,\mathbb{R}) \begin{array}{l}
\underleftarrow{\delta_{i-1}}\\ 
 \overrightarrow{ \delta^{*}_{i-1}}
\end{array}
 C^{i-1}(K,\mathbb{R}),
\end{equation*}
all three operators are well defined. Moreover, it follows directly from the definition  that $\mathcal{L}_{i}^{up}(K)$, $\mathcal{L}_{i}^{down}(K)$ and  $\mathcal{L}_{i}(K)$ are self-adjoint, non-negative and compact operators.
Hence  their eigenvalues are real, non-negative, and can be characterized by the Courant-Fischer-Weyl min-max principle.
\begin{thm}[Min-max theorem]
\label{min-max theorem}
Let $\mathcal{V}_{k}$  denote a $k$-dimensional subspace  of $V$, and assume $A:V\rightarrow V$ is a compact, self adjoint operator of a Hilbert space $V$. Then
\begin{equation}
\label{min-max equation}
\lambda_{k}=\min_{\mathcal{V}_{k}}\max_{g\in V_{k}}\frac{(Ag,g)}{(g,g)}=\max_{\mathcal{V}_{m-k+1}}\min_{g\in \mathcal{V}_{m-k+1}}\frac{(Ag,g)}{(g,g)},
\end{equation}
where $\lambda_{1}\leq \ldots \leq \lambda_{m}$ are the eigenvalues of  $A$.
\end{thm}
For any operator $A$ acting on a Hilbert space, we denote the weakly increasing rearrangement of its eigenvalues by $\s(A)=(\lambda_{0},\ldots,\lambda_{m})$ and write 
 $\s(A)\dote \s(B)$, when the multisets $\s(A)$ and $ \s(B)$ differ
 only in their multiplicities of zero (this is is an equivalence relation). We denote a union of multisets by  $\dotu$. 

We now state 
 the discrete version of the Hodge theorem and provide its proof for the sake of completeness.
\begin{thm}[Eckmann 1944 ]
 For an abstract simplicial complex $K$, 
 \begin{displaymath}
  \ker \mathcal{L}_{i}(K)\cong \tilde{H}^{i}(K,\rb).
 \end{displaymath}
\end{thm}
\begin{proof}
Since $\delta_{i}\delta_{i-1}=0$ and $\delta_{i-1}^{*}\delta_{i}^{*}=0$, then 
\begin{align}
&\im\mathcal{L}_{i}^{down}(K) \subset  \ker\mathcal{L}_{i}^{up}(K)\label{hodge},\\
&\im \mathcal{L}_{i}^{up}(K) \subset  \ker \mathcal{L}_{i}^{down}(K)\label{hodge1}.
\end{align}
Thus,
\begin{align*}
  \ker \mathcal{L}_{i}(K)&=\ker\delta^{*}_{i}\delta_{i}\cap   \ker\delta_{i-1}\delta^{*}_{i-1}\\
&= \ker \delta_{i} \cap \ker \delta^{*}_{i-1}\\
&=  \ker \delta_{i} \cap (\im \delta_{i-1})^{\perp}\\
&\cong \tilde{H}^{i}(K,\rb).
\end{align*}
\end{proof}
Due to (\ref{hodge}) and  (\ref{hodge1}),
$\lambda$ is a non-zero eigenvalue of $\mathcal{L}_{i}(K)$ if and only if it is an eigenvalue of $\mathcal{L}^{up}_{i}(K)$ or  $\mathcal{L}^{down}_{i}(K)$.
Therefore,
\begin{equation}
\label{com1}
\s(\mathcal{L}_{i}(K))\dote \s(\mathcal{L}_{i}^{up}(K))\dotu \s(\mathcal{L}_{i}^{down}(K)).
\end{equation}
As a direct consequence of the fact that $\s(AB)\dote \s(BA)$, for
operators $A$ and $ B$  on suitably chosen Hilbert spaces, we get the
following equality, which was pointed out to us by Johannes Rauh. 
\begin{equation}
\label{com2}
\s(\mathcal{L}_{i}^{up}(K))\dote\s(\mathcal{L}_{i+1}^{down}(K)).
\end{equation}
From (\ref{com1}) and (\ref{com2}) we conclude that each of  the three families of multisets 
$$\{\s(\mathcal{L}_{i}(K))\mid -1\leq i \leq d\}\textrm{,}\; \{\s(\mathcal{L}_{i}^{up}(K))\mid -1\leq i \leq d\} \;\textrm{or}\; \{\s(\mathcal{L}_{i}^{down}(K))\mid 0\leq i \leq d\}$$
determines the other two. Therefore, it suffices to consider only one
of them.
In the remainder of the paper, we will omit the argument $K$ in $\s(\mathcal{L}_{i}(K))$, 
$\s(\mathcal{L}_{i}^{up}(K))$, $\mathcal{L}_{i}^{up}(K)$, $\mathcal{L}_{i}^{down}(K)$, $S_{i}(K)$ etc  when   
it is clear which simplicial complex we investigate or when we state our results for a general simplicial complex $K$.

For explicit expressions for up and down Laplacians, we have to fix  scalar products on the cochain groups. To that end, we introduce the weight function and additional notation.
\begin{defn}
The  \emph{ weight function} $w$ is an evaluation function on the set of all faces of $K$
 $$w:\bigcup_{i=-1}^{\dim K}S_{i}(K)\rightarrow \rb^{+}.$$ 
The \emph{weight of a face} $F$ is $w(F)$.
\end{defn}
For any choice of the inner product on  the space  $C^{i}(K,\rb)$, where elementary cochains form an orthogonal basis,  there exists a weight function $w$, such that
$$
(f,g)_{C^{i}}=\sum_{F\in S_{i}(K)} w(F)f([F])g([F]).
$$ 
Furthermore, there is  a one-to-one correspondence between   weight functions  and possible scalar products on cochain groups $C^{i}(K,\rb)$, such that elementary cochains are orthogonal. In the remainder we will interchangeably use the terms weights, weight function and scalar product.
\begin{defn}
Let $\bar{F}=\{v_{0},\ldots, v_{i+1}\}$ be an $(i+1)$-face of a complex $K$ and let $F=\{v_{0},\ldots,\hat{v_{k}},\ldots,v_{i+1}\}$ be an $i$-face of $\bar{F}$.   Then \emph{the boundary of the oriented face} $[\bar{F}]$ is
$$
\partial [\bar{F}]= \sum_{k}(-1)^{k}[v_{0},\ldots,\hat{v_{k}},\ldots,v_{i+1}],
$$
and the sign of  $[F]$ in the boundary of $[\bar{F}]$ is denoted by $\sgn([F],\partial [\bar{F}])$ and is equal to $(-1)^{k}$.
 \end{defn}
By abuse of notation,   we will  write $\partial \bar{F}$ to denote  the set of all $i$-faces of $\bar{F}$.
The  $i$-up Laplace operator is given by
\small
$$
(\mathcal{L}_{i}^{up} f)([F])= \sum_{\substack{\bar{F}\in S_{i+1}:\\ F\in\partial \bar{F}}} \frac{ w(\bar{F})}{w(F)} f([F]) 
+   \sum_{\substack{F'\in S_{i}: F\neq F',\\ F,F'\in\partial  \bar{F}}} \frac{w(\bar{F})}{w(F)}\sgn([F],\partial [\bar{F}])\sgn([F'],\partial [\bar{F}])f([F']),
$$
\normalsize
and  the expression for the $i$-down  Laplace operator is 
\small
$$
(\mathcal{L}^{down}_{i} f)([F])= \sum_{E\in \partial F}\frac{w(F)}{w(E)}f([F])+ \sum_{F': F\cap F'=E}\frac{w(F')}{w(E)}
 \sgn([E],\partial [F])\sgn([E],\partial [F']) f([F']).
$$
\normalsize
When dealing with linear operators  it is often more 
convenient to study their matrix form. Hence we give the following expressions for  
 the  $(e_{[F]},e_{[F']})$-th and the $(e_{[F]},e_{[F]})$-th entry of $\mathcal{L}^{up}_{i}$ and $\mathcal{L}^{down}_{i}$, where   $F\neq F'$
\begin{align*}
 & (\mathcal{L}^{up}_{i} )_{(e_{[F]},e_{[F']})}=\sgn([F],\partial [\bar{F}])\sgn([F'],\partial [\bar{F}]) \frac{w(\bar{F})}{w(F)},\\
& (\mathcal{L}^{up}_{i} )_{(e_{[F]},e_{[F]})}=\sum_{\substack{\bar{F}\in S_{i+1},\\ F\in\partial \bar{F}}} \frac{ w(\bar{F})}{w(F)},\\
& (\mathcal{L}^{down}_{i} )_{(e_{[F]},e_{[F']})}=\sgn([E],\partial [F])\sgn([E],\partial [F']) \frac{w(F')}{w(E)},\\
 &(\mathcal{L}^{down}_{i} )_{(e_{[F]},e_{[F]})}=\sum_{E\in \partial F}\frac{w(F)}{w(E)}.
\end{align*}
The Laplace operator $\mathcal{L}$  of a simplicial complex $K$ is  uniquely determined by a weight function $w_{K}$ on the faces of $K$.
Thus, we write $\mathcal{L}(K,w_{k})$.
\begin{rem}
\label{pure sc dependencies}
From the explicit expressions of Laplace operators it is evident that 
$\mathcal{L}^{up}_{i}$ is uniquely determined by its restriction on
the $(i+1)$-skeleton of  $K$, whereas $\mathcal{L}^{down}_{i}$ is determined by its  $i$-skeleton.
Therefore, when studying  $\mathcal{L}^{up}_{i}$ (or $\mathcal{L}^{down}_{i}$), it suffices to observe  pure $(i+1)$(or $i$)-simplicial complexes.
\end{rem}
Let $D_{i}$  be the matrix corresponding to the operator $\delta_{i}$, $D_{i}^{T}$ its  transpose and
$W_{i}$ the diagonal matrix representing the scalar product on $C^{i}$, then the $\mathcal{L}^{up}_{i}$  and $\mathcal{L}^{down}_{i}$ operators are expressed as
$$
\mathcal{L}^{up}_{i}=W_{i}^{-1}D_{i}^{T}W_{i+1}D_{i},
$$
and
$$
\mathcal{L}^{down}_{i}=D_{i-1}W_{i-1}^{-1}D_{i-1}^{T}W_{i},
$$
respectively.
Therefore,  the combinatorial Laplace operator analysed by Duval,
Reiner \cite{DuvalShifted}, Friedmann \cite{Friedman} and others
\cite{Muhammad},\cite{Dong}, is  the combinatorial Laplace operator
$\mathcal{L}_{i}$ for the identity matricesas weight matrices  $W_{i}$ ($-1\leq i \leq \dim K$), i.e. $\mathcal{L}_{i}(K,w_{K})$, where $w_{K}\equiv 1$.
In the remainder of the paper, this version of 
the Laplace operator will be denoted by  $L_{i}$. The  graph Laplacian  (\ref{GL}) studied by Kirchhoff \cite{Kirchhoff}, Fiedler \cite{Fiedler1973}, Grone and Merris \cite{GroneMerris} and  many others is a special case of $L_{i}$, in fact it is equal to $L_{0}^{up}$.
The \emph{normalized graph Laplace operator} (\ref{NGL}) investigated
by  Chung, Yau, Grigoryan and others, see \cite{Chung1996} and
\cite{Jost},   is equal to $\mathcal{L}^{up}_{0}$ for  $W_{1}$  being
the  matrix with diagonal entries equal to the edge weights and
$W_{0}$ the diagonal degree matrix, that is 
the weight function on a vertex $v$ is $w(v)=\deg v$.

Therefore, the combinatorial Laplacian $\mathcal{L}(K,w_{K})$, as
defined here, unifies all Laplace operators studied so far and
provides a general framework for a systematic study of different versions of Laplacians. 

Our goal in this paper is to  define  the  higher dimensional analogue
of the normalized graph Laplacian and to investigate its properties. 
However, we will  state our results in full generality  whenever possible,  and emphasize which results do not depend  on the choice of the scalar products, and  which  are the consequence
of suitably chosen weights.

 \section{The normalized combinatorial Laplacian: definition and basic properties}
In this section we derive an upper and a lower bound for the maximal
eigenvalue of $\mathcal{L}^{up}_{i}$, 
introduce the normalized combinatorial Laplacian $\lup$, and state and prove its basic properties.
We emphasize its advantages compared to the other choices of weights.

Let $\lambda_{m}$ and $\lambda_{0}$   be the maximl and the minimal
eigenvalue  of  $\mathcal{L}_{i}^{up}(K,w_{K})$, respectively. As  the
Laplace operator is positive definite,  $\lambda_{0}$ is always larger or equal to zero. 
The exact number of zero eigenvalues in the spectrum of 
$\mathcal{L}_{i}^{up}$ and $\mathcal{L}_{i}^{down}$ is given in the following  theorem.
\begin{thm}
\label{thm zeros}
 The multiplicity of the eigenvalue zero in 
\begin{itemize}
\item[(i)]  $\s(\mathcal{L}_{i}^{up})$ is 
 \begin{equation*}
 \dim C^{i}-\sum_{j=0}^{i}(-1)^{i+j}(\dim C^{j}-\dim \tilde{H}^{j}),
\end{equation*}
or equivalently
 \begin{equation*}
 \label{other zeros}
\dim C^{i} + \sum_{j=1}^{d-i}(-1)^{j}(\dim C^{i+j}-\dim \tilde{H}^{i+j}).
\end{equation*}
\item[(ii)] $\s(\mathcal{L}_{i}^{down})$ is
 \begin{equation*}
\dim \tilde{H}^{i}-\sum_{j=0}^{i-1}(-1)^{i+j-1}(\dim C^{j}-\dim \tilde{H}^{j}).
\end{equation*}
\end{itemize} 
\end{thm}
\begin{proof}
The following  are short exact sequences  that split
\begin{displaymath}
 0\rightarrow\ker\delta_{i}\rightarrow C^{i}\rightarrow \im\delta_{i}\rightarrow 0,
\end{displaymath}
\begin{displaymath}
 0\rightarrow\im\delta_{i-1}\rightarrow\ker\delta_{i}\rightarrow \tilde{H}^{i}\rightarrow 0.
\end{displaymath}
This is a direct consequence of   the fact that $\im\delta_{i}$ and $\tilde{H}^{i}$ are projective modules (for details on projective modules and 
splitting exact sequences the reader is referred to \cite{Cohn}).
Therefore,
\begin{equation}
\label{zeros Ci}
 \dim C^{i}= \dim \ker\delta_{i} + \dim \im\delta_{i},
\end{equation}
and
\begin{equation}
\label{zeros keri}
\dim \ker\delta_{i} =  \dim \tilde{H}^{i}+ \dim \im\delta_{i-1}.
\end{equation}
From (\ref{zeros Ci}) and (\ref{zeros keri}) 
$$
\dim\im \delta_{i}=\sum_{j=0}^{i}(-1)^{i+j}(\dim C^{j}-\dim \tilde{H}^{j}).
$$
The number of zeros in the spectrum of $\mathcal{L}_{i}^{up}$ is equal to  the dimension of its kernel, thus
\begin{align*}
 \dim \ker \mathcal{L}_{i}^{up}&=  \dim \ker \delta_{i}\\
&=  \dim C^{i} - \sum_{j=0}^{i}(-1)^{j+i}(\dim C^{j}-\dim \tilde{H}^{j}).
\end{align*}
The expression (\ref{other zeros}) for the number of zeros  in $\s(\mathcal{L}_{i}^{up})$ is easily obtained
 by using the Euler characteristic and the equality 
$\chi=\sum_{j=-1}^{d}(-1)^{i}\dim C^{i}=\sum_{j=-1}^{d}(-1)^{i}\dim \tilde{H}^{i}$.
As for   $\mathcal{L}_{i}^{down}$, the following holds
 \begin{align*}
 \dim \ker \mathcal{L}_{i}^{down}&=  \dim \ker \delta^{*}_{i-1}=\dim C^{i}-\dim\im \delta_{i-1}\\
&=  \dim C^{i} -\sum_{j=0}^{i-1}(-1)^{j+i-1}(\dim C^{j}-\dim \tilde{H}^{j}).
\end{align*}
\end{proof}
The number of zero eigenvalues  in   spectra of various Laplace
operators, as expected, does not depend on a choice of the scalar
products on the cochain vector spaces.
\begin{rem}
If a simplicial complex is $(i+1)$-dimensional, then the number of zero eigenvalues in the spectrum of 
$\mathcal{L}^{up}_{i}(K)$ is $\dim C^{i}-\dim C^{i+1}+\dim \tilde{H}^{i+1}$, whereas 
there are exactly $\dim C^{i+1}-\dim \tilde{H}^{i+1}+\dim \tilde{H}^{i}$ zeros in the spectrum of $\mathcal{L}^{down}_{i+1}$.
\end{rem}

Next we introduce the  degree of  a simplex $F$.
\begin{defn}
\emph{The degree } of an $i$-face $F$ of $K$ is equal to the sum of
the weights of all simplices that contain $F$ in its boundary, i.e.
\begin{equation*}
\deg F= \sum_{\bar{F}\in S_{i+1}(K): F\in \partial \bar{F}}w(\bar{F}).
\end{equation*}
\end{defn}

The upper bound on $\s(\mathcal{L}_{i}^{up})$ follows from the subsequent discussion.
We have
\begin{subequations} \begin{align}
(\mathcal{L}^{up}_{i} f,f)&= (\delta_{i}f,\delta_{i}f)\\
&=(\sum_{\bar{F}\in S_{i+1}(K)}f(\partial [\bar{F}])\e_{[\bar{F}] },\sum_{\bar{F}\in S_{i+1}(K)}f(\partial [\bar{F}])\e_{[\bar{F}] })\\
&=\sum_{\bar{F}\in S_{i+1}(K)}f(\partial [\bar{F}])^{2}w(\bar{F})\\
\label{last} &\leq (i+2)\sum_{F\in S_{i}(K)} f([F])^{2}\sum_{\bar{F}\in S_{i+1}(K):F\in \partial \bar{F}}w(\bar{F}),
\end{align}
\end{subequations}
where (\ref{last}) is obtained by using  the Cauchy-Schwarz inequality.
In terms of degrees the  last inequality  can be restated  as
\begin{equation}
\label{ineq0}
(\mathcal{L}^{up}_{i} f,f) \leq (i+2)\sum_{F\in S_{i}(K)} f([F])^{2}\deg F.
\end{equation}
By dividing   (\ref{ineq0}) by$(f,f)$ we get
\begin{equation}
\label{ineq1}
\frac{(\mathcal{L}^{up}_{i} f,f)}{(f,f)}\leq (i+2)\frac{\sum_{F\in S_{i}(K)} f([F])^{2}\deg F }{\sum_{F\in S_{i}(K)}f([F])^{2}w( F)}.
\end{equation}
Replacing $f$ in (\ref{ineq1}) with the eigenfunction $f_{m}$, corresponding to the largest eigenvalue $\lambda_{max}$ of $\mathcal{L}_{i}^{up}$ gives
\begin{equation}
\label{ineq2}
\lambda_{max}\leq (i+2)\frac{\sum_{F\in S_{i}(K)} f_{m}([F])^{2}\deg F }{\sum_{F\in S_{i}(K)}f_{m}([F])^{2}w(F)}.
\end{equation}
Therefore,   if
\begin{equation}
\label{normalizing condition}
w(F)=\deg F,
\end{equation}
for every $F\in S_{i}(K)$,
then $\lambda_{max}\leq i+2$ and the eigenvalues of $\mathcal{L}^{up}_{i}$ are in the interval $[0,i+2]$.
\begin{defn}
Let $w$ be a weight function on $K$ which satisfies (\ref{normalizing condition}) for every face of simplicial complex $K$, which is not a  facet ($\dim F<\dim K$), then 
 the Laplace operator defined  on the cochain complex of $K$  is called 
the \emph{weighted normalized combinatorial Laplace operator}.
If additionally, the weights of the facets of $K$ are equal to $1$, then  the obtained operator is called  
the \emph{ normalized combinatorial Laplace operator} and is denoted
by $\lup$. We will keep the same notation for the weighted normalized
combinatorial Laplacian, emphasizing that we are considering its weighted version.
\end{defn}
If  (\ref{normalizing condition}) does not hold, 
 we  derive a bound on the maximal eigenvalue of the  Laplacian $\mathcal{L}^{up}_{i}$  from the inequality (\ref{ineq2}), i.e.
\begin{equation}
\label{weighted max bound}
\lambda_{m} \leq  (i+2) \frac{\max_{F\in S_{i}(K)}\deg F}{\min_{F\in S_{i}(K)}w(F)}.
\end{equation}
Here  $\min_{F\in S_{i}(K)}w(F)$ stands for the minimal \emph{non-zero} weight over all $i$-faces $F$ of  $K$.
The  inequality (\ref{weighted max bound}) in the case of the combinatorial Laplacian $L_{i}^{up}$    reduces to 
\begin{equation}
\label{bound}
\lambda_{m} \leq (i+2) \max_{F\in S_{i}(K)}\deg F,
\end{equation} 
which for $i=0$  becomes exactly  
$$
\lambda_{m}\leq 2 \max_{v\in S_{0}(G)}\deg v.\\
$$
 This is the well-known bound    on the maximal eigenvalue of $L_{0}^{up}$ (see \cite{Anderson}).
Another upper bound of the spectrum of $L_{i}^{up}$ was obtained by  Duval and Reiner in \cite{DuvalShifted}  as a part of more general study, i.e.
\begin{equation}
\label{Duval max est}
\lambda_{m}\leq n,
\end{equation}
where $n$ is the number of vertices of the  complex $K$.
The inequality (\ref{bound}) is  sharper than (\ref{Duval max est}) for large values of $n$ and small values of $i$. In particular, 
if $\max_{F\in S_{i}}\deg F < n /(i+2)$, then the estimate (\ref{bound}) is sharper, otherwise it is  (\ref{Duval max est}).
We sum up our results in the following theorem.
\begin{thm}
The  spectrum of $\mathcal{L}_{i}^{up}$ is bounded from above by: 
\begin{itemize}
\item[(i)] $i+2$, if  $\mathcal{L}_{i}^{up}=\lup$,
\item[(ii)] $(i+2) \max_{F\in S_{i}(K)}\deg F$, if  $\mathcal{L}_{i}^{up}=L_{i}^{up}$,
\item[(iii)] $(i+2) \max_{F\in S_{i}(K)}\deg F/\min_{F\in S_{i}(K)}w(F)$, for all other choices of  scalar products.
\end{itemize}
\end{thm}

In the following theorems we present some lower bounds on $\lambda_{max}$.
\begin{thm}
Without  loss of generality, let $K$ be a pure simplicial complex of
dimension $i+1$,  $\mathcal{L}(K,w_{K})$ the Laplace operator with 
the weight function $w_{K}$,  $\lambda_{max}$ the maximal eigenvalue in the spectrum of $\mathcal{L}_{i}^{up}(K,w_{K})$,  and 
$\vol_{i}(K)=\sum_{F\in S_{i}} \deg F$, then
\begin{itemize}
\item[(i)] $\dim C^{i}/(\dim C^{i+1}-\dim H^{i+1})\leq \lambda_{max}$, if  $\mathcal{L}_{i}^{up}=\lup$,
\item[(ii)] $ \vol_{i}(K) /(\dim C^{i+1}-\dim H^{i+1})\leq \lambda_{max}$, if  $\mathcal{L}_{i}^{up}=L_{i}^{up}$,
\item[(iii)] $\vol_{i}(K) /(\max_{F\in S_{i}} w(F)(\dim C^{i+1}-\dim H^{i+1}))\leq \lambda_{max}$, for all other choices of  scalar products.
\end{itemize}
\end{thm}
\begin{proof}
The sum of all eigenvalues is equal to the trace of  the Laplace matrix, i.e. $\sum_{F\in S_{i}}\sum_{\bar{F}: F\in \bar{F}}\frac{w(\bar{F})}{w(F)}$. Together with Theorem \ref{thm zeros}, this yields the inequality
\begin{equation}
 \frac{\sum_{F\in S_{i}}\sum_{\bar{F}: F\in \bar{F}}\frac{w(\bar{F})}{w(F)} }{\dim C^{i+1}-\dim H^{i+1}}\leq \lambda_{max},
\end{equation}
which proves the theorem.
\end{proof}

\begin{thm}
\label{lower bound max general}
Let $K$ be a pure $(i+1)$-dimensional  simplicial complex and let $\lambda_{max}$ denote the maximum eigenvalue of the operator $\mathcal{L}^{up}_{i}(K,w)$,   then 
\begin{equation}
\frac{D}{d}+ \frac{(i+1)D}{Nd} \leq \lambda_{max},
\end{equation}
where $D$, $d$ are  maximal degree, weight, respectively  over all $i$-simplices and  $N$ is the minimal number of $(i+1)$-faces which are incident to an $i$-simplex of degree $D$.
\end{thm}
\begin{proof}
Assume $F$ is an $i$-simplex of maximal degree with the minimal number of incident $(i+1)$-faces, i.e. there exist exactly  $N$ $(i+1)$-simplices which contain $F$ as a facet and $\sum_{\bar{F}:F\in \bar{F}}w(\bar{F})=D$.
Let $f=\sum_{k=1}^{N}\sgn([F],\partial[\bar{F_{k}}])e_{[\bar{F_{k}}]}$, then we obtain
\small
\begin{align*}
\lambda_{max} &\geq \frac{(\delta^{*}_{i}f,\delta^{*}_{i}f)}{(f,f)} \\
&\geq \!\! \frac{1}{D}\!\! \left(\!\sum_{k=1}^{N}\sgn([F],\partial[\bar{F_{k}}])\!\!\!\!\!\sum_{E\in\partial \bar{F_{k}}}\!\!\! \sgn([E],\partial  [\bar{F_{k}}])\frac{w(\bar{F}_{k})}{w(E)}\e_{[E]},  \!\!\sum_{k=1}^{N}\sgn([F],\partial [\bar{F_{k}}])\!\!\!\!\!\sum_{E\in\partial \bar{F_{k}}}\!\!\! \sgn([E],\partial  [\bar{F_{k}}])\frac{ w(\bar{F}_{k})}{w(E)} \e_{[E]}\right)   \\
&= \frac{1}{D}  ( e_{[F]},  e_{[F]})  \\
&+\!\!  \frac{1}{D}\!\! \left( \sum_{k=1}^{N}\sgn([F],\partial[\bar{F_{k}}])\!\!\!\!\! \sum_{\substack{E\in\partial \bar{F_{k}}:\\ E\neq F}}\!\!\! \sgn([E],\partial  [\bar{F_{k}}]) \frac{w(\bar{F}_{k})}{w(E)} \e_{[E]},  \!\sum_{k=1}^{D}\sgn([F],\partial [\bar{F_{k}}])\!\!\!\!\! \sum_{\substack{E\in\partial \bar{F_{k}}:\\ E\neq F}}\!\!\! \sgn([E],\partial  [\bar{F_{k}}])\frac{w(\bar{F}_{k}) }{w(E)}\e_{[E]}\right) \\
&= \frac{D}{w(F)}+ \frac{1}{D}\sum_{k=1}^{N}\sum_{\substack{E\in\partial \bar{F_{k}}:\\ E\neq F}} \frac{w^{2}(\bar{F_{k}})}{w(E)}\\
&\geq \frac{D}{w(F)}+ \frac{i+1}{dD}\sum_{k=1}^{N}w^{2}(\bar{F_{k}})\\
&\geq \frac{D}{d}+ \frac{i+1}{dD}\frac{D^{2}}{N}
\end{align*} 
\normalsize
The inequalities above are  a consequence of the variational characterization of eigenvalues (Theorem \ref{min-max theorem}), and of the Cauchy-Schwarz inequality.
\end{proof}
The previous result for 
$\mathcal{L}(K,w_{K})=L$ generalizes Proposition 8.2. from
\cite{DuvalShifted} and its proof. 
As another special case of Theorem \ref{lower bound max general} we obtain the following lower bounds for the maximal eigenvalue of the normalized Laplacian.
\begin{coll}
\begin{equation}
1+ \frac{i+1}{D} \leq \lambda_{max},
\end{equation}
where $D$ is the maximal degree over all $i$ simplices, and $\lambda_{max}$ the maximal eigenvalue of $\lup$.
\end{coll}

\begin{rem}[Negative Weights]
If negative weights in the definition of the weight function  are
allowed , then  bilinear forms (inner products) on cochain vector
spaces are no longer positive definite.
With arbitrary weights,  $\mathcal{L}_{i}^{up}$ acts on functions on $i$-simplices 
$$
\lup f([F])=\frac{1}{w(F)} \sum_{\substack{\bar{F}\in S_{i+1}\\F\in\partial\bar{F}}}\sgn([F],\partial \bar{[F]})f(\partial \bar{[F]}).
$$
This approach enables us to use negative weights, but it also deprives
us of the structure of the cohomology of simplicial complexes. The
eigenvalues need no longer be real nor non-negative.
Here, however, we do not pursue the study of Laplacians with negative weights.
\end{rem}

\section{Circuits, paths, stars and their spectrum }
In this section we calculate the   spectrum  of the up (down) normalized Laplace operator for some 
 classes of simplicial  complexes.  
\begin{thm}
Let $K$ be an  $(n-1)$-dimensional simplex. Then $\s(\lup(K))$  consists of
the eigenvalue $ n/(n-i-1)$ with multiplicity $\binom{n-1}{i+1}$ and
the eigenvalue zero with multiplicity $\binom{n-1}{i}$.
\end{thm}
\begin{proof}
 We will  prove that the function $f\in C^{i}(K,\rb)$, 
\begin{displaymath}
f_{[\bar{F}]}([F])=\left\{ \begin{array}{cl}
\sgn([F],\partial \bar{[F]}) & \textrm{if } F \textrm{ is facet of  $(i+1)$-face } \bar{F} \\
0& \textrm{ otherwise,}\\
\end{array}
\right.
\end{displaymath}
is an eigenfunction of $\lup(K)$ for the eigenvalue $ n/(n-i-1)$.\\
It is not difficult to see that there are exactly $ \binom{n-1}{i+1}$  linearly independent functions of this form.
We have to check that the equality 
$$
(\lup f_{[\bar{F}]})[F]=\frac{n}{n-i-1} f([F])
$$
holds for every $i$-dimensional face $F$ of $K$. We distinguish three cases:\\
$(i)$ $F$ is an arbitrary facet of $\bar{F}$. Then,\small
\begin{align*}
(\Delta_{i}^{up} f_{[\bar{F}]})([F])={}&  \sum_{\substack{\bar{E}\in S_{i+1}: \\F\in \partial \bar{E}}} \frac{w(\bar{E})}{w(F)} f_{[\bar{F}]}([F]) \\
 & + \sum_{\substack{F'\in S_{i}(L):\\ (\exists \bar{E}\in S_{i+1}(L)) F,F'\in  \partial \bar{E}}} 
\frac{w(\bar{E})}{w(F)}\sgn([F],\partial \bar{[E]})\sgn([F'],\partial \bar{[E]})f_{[\bar{F}]}([F'])\\
={} &\frac{1}{n-i-1} \sum_{\substack{\bar{E}\in S_{i+1}: \\F\in \partial \bar{E}}}f_{[\bar{F}]}(F)\\
&+\frac{1}{n-i-1} \sum_{\substack{F'\in S_{i}(L):\\ (\exists \bar{E}\in S_{i+1}(L)) F,F'\in  \partial \bar{E}}}\sgn([F], \partial \bar{[E]})\sgn([F'],\partial \bar{[E]})f_{[\bar{F}]}([F'])\\
={}& f_{[\bar{F}]}([F])+ \frac{i+1}{n-i-1}\sgn([F],\partial \bar{[F]})\\
={}& \frac{n}{n-i-1}f([F]).
\end{align*}
\normalsize
$(ii)$  $F$ and $\bar{F}$ have $i$ vertices in common, i.e. they intersect in a face of dimension $i-1$. \\
Then by definition   $f([F])=0$. 
Let $v_{0}, v_{1}, \ldots, v_{i+2}\in [n]$ be arbitrary vertices of $L$ ordered increasingly. 
Without loss of generality,  assume 
$0\leq j<k<l\leq i+2$, 
 $\bar{F}=[v_{0},\ldots,\hat{v_{l}},\ldots, v_{i+2}]$ and 
$[F]=[v_{0},\ldots,\hat{v_{j}}, \ldots, \hat{v_{k}},\ldots, v_{i+2}]$.
Then there exist exactly  two $i$-faces  $F_{1}$ and $F_{2}$  in the boundary of $\bar{F}$ and two  
$(i+1)$-simplices $\bar{F}_{1}$ and  $\bar{F}_{2}$ of $L$, such that $F, F_{1} \in \partial \bar{F}_{1} \textrm{ and } F, F_{2} \in \partial \bar{F}_{2}$.
In particular, $F_{1}=[v_{0}, \ldots, \hat{v_{k}},\ldots,\hat{v_{l}},\ldots, v_{i+2}]$,  
$F_{2}=[v_{0},\ldots,\hat{v_{j}}, \ldots,\hat{v_{l}},\ldots, v_{i+2}]$  and 
$\bar{F}_{1}=[v_{0}, \ldots, \hat{v_{k}},\ldots, v_{i+2}]$,  $\bar{F}_{2}=[v_{0},\ldots,\hat{v_{j}}, \ldots, v_{i+2}]$.
Now it is straightforward to calculate
\begin{align*}
(\Delta_{i}^{up} f_{[\bar{F}]})([F])={} &  0+  \sgn([F],\partial [\bar{F}_{1}])\sgn([F_{1}],\partial[\bar{F}_{1}])f_{[\bar{F}]})([F_{1}])\\
& +
\sgn([F],\partial [\bar{F}_{2}])\sgn([F_{2}],\partial[\bar{F}_{2}])f_{[\bar{F}]})([F_{2}])\\
={} & \sgn([F],\partial [\bar{F}_{1}])\sgn([F_{1}],\partial[\bar{F}_{1}])\sgn([F_{1}],\partial \bar{[F]})\\
&+
\sgn([F],\partial [\bar{F}_{2}])\sgn([F_{2}],\partial[\bar{F}_{2}])\sgn([F_{2}],\partial \bar{[F]})\\
= {} &  (-1)^{j}(-1)^{l-1}(-1)^{k}+ (-1)^{k-1}(-1)^{l-1}(-1)^{j}\\
={} &0.
\end{align*}
$(iii)$  $F$ and $\bar{F}$ have less than $i$ vertices in common. \\Then there are no faces in the boundary of $\bar{F}$ which are $(i+1)$-up neighbours of $F$.
This implies  that $\lup f([F])=0$, which  completes the proof.
\end{proof}
In the remainder of this section, we calculate the spectrum of circuits, paths and stars.
\begin{defn}
\label{circuit def}
A pure simplicial complex $L$ of dimension $i$ is called an \emph{$i$-path of length $m$} iff
 there is an ordering of its $i$-simplices $F_{1}< F_{2}<\ldots <F_{m}$, such that 
 $F_{i}$ and $F_{j}$ are $(i-1)$-down neighbours iff $ \mid j-l\mid =1$.
When $F_{m}$ coincides with $F_{1}$,    we say that $L$ is an \emph{$i$-circuit}
of length $(m-1)$. The vertices in the intersection   $\bigcap_{j=1}^{m-1}F_{j}$  are called \emph{centers} of  $L$.
\end{defn}

\begin{figure}[h!tp]
  \begin{center}
 \subfigure[raggedright,scriptsize][$2$-circuit of length $6$ with an empty center]{\label{circuit empty}\includegraphics[width=2cm]{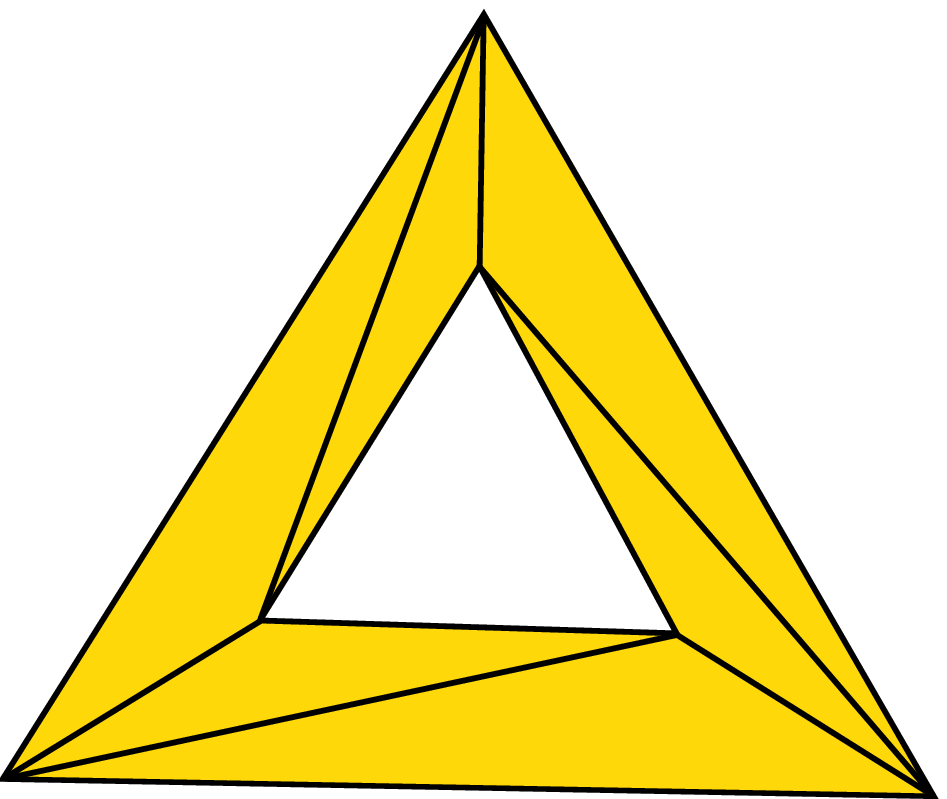}} \qquad
    \subfigure[raggedright,scriptsize][$2$-circuit of  length $6$ with one vertex in a  center]{\label{circuit center}\includegraphics[width=2cm]{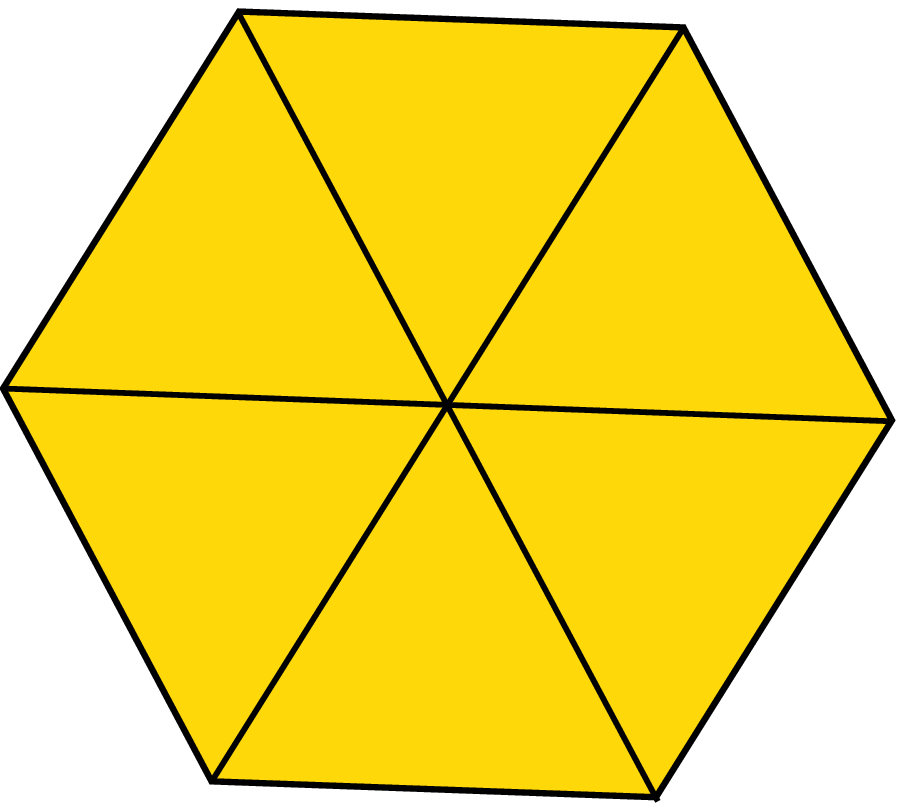}}\qquad
    \subfigure[raggedright][$2$-path of length  $3$]{\label{path center}\includegraphics[width=2cm]{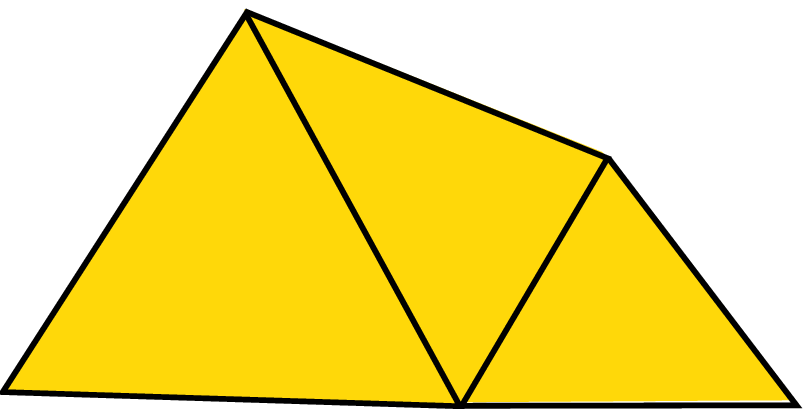}}\qquad
    \subfigure[raggedright][$2$-star of  length  $3$]{\label{star}\includegraphics[width=2cm]{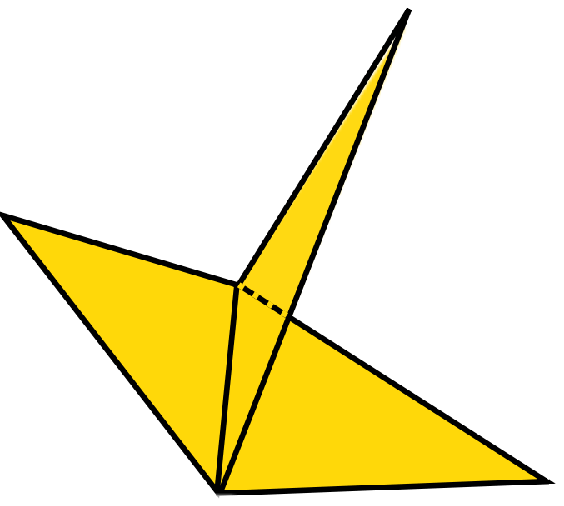}}
  \end{center}
  \caption{Examples of circuits, paths and stars}
\end{figure}
Note that the simplicial complexes in Figures \ref{circuit center} and \ref{path center} have one central vertex, i.e. a center.\\
Before we proceed to calculate $\s(\lup)$ of these complexes, we  recall the definition of orientability.
 \begin{defn}
Let $K$ be a pure $(i+1)$-dimensional simplicial complex. We say that $K$ is \emph{orientable} iff
it is possible to assign an orientation to all $(i+1)$-faces of $K$ in such a way that any two simplices that intersect in an $i$-face induce a different orientation on that face.
We say that such simplices are oriented \emph{coherently}.
\end{defn}
Note that if an $i+1$-dimensional simplicial complex is orientable,
then any of its $i+1$-faces has at most one
$i$-down neighbour.

Choosing an orientation on $(i+1)$-faces of the orientable simplicial
complex $K$ is equivalent to choosing a basis $B_{i+1}(K)$ of the
vector space $C_{i+1}(K,\rb)$ consisting of elementary  $(i+1)$-chains
$[\bar{F}]$ that are oriented \emph{coherently}.

For the subsequent calculations, the following obvious result (see
e.g.   \cite{greenman}) will be useful.
\begin{lem}
\label{Greenman1}
If two matrices $M$ and $P$ commute, i.e., $MP=PM$, and if $\lambda$ is a simple eigenvalue of $P$, then its corresponding eigenvector $v$ is also an eigenvector of $M$.
\end{lem}

Let $\tilde{p}$ be  a permutation of  the elements of  a basis $B_{i}(K)$ of  $C_{i}(K,\rb)$, for an arbitrary simplicial complex $K$, and let   $\bar{p}$ be the permutation of elementary cochains  of dimension $i$ induced by $\tilde{p}$. Denote  the   linear extension  of $\bar{p}$ on 
$C^{i}(K,\rb)$ by  $p$.
 Then we have the following equivalences
$$\tilde{p}([F])=[F‘]\Leftrightarrow \bar{p}(e_{[F]})=e_{[F‘]} \Leftrightarrow p(e_{[F]})=e_{[F‘]} .$$
To simplify the notation, we will designate any of the maps $\tilde{p}$, $\bar{p}$, $p$ by $p$.
 It will be clear from the argument of $p$ which one is used.
Furthermore, we will write $p(F)$ to denote the $i$-face which is uniquely determined by the mapping $p([F])$.
To prove that $p$ and $\ldown$ commute, it is necessary to check  if $p\Delta_{i}^{down} e_{[F]}=\Delta_{i}^{down} pe_{[F]}$ holds for every 
$i$-face $F$.
Since 
\small
\begin{align*}
p\Delta_{i}^{down} e_{[F]}={}&
\sum_{E\in \partial F} \frac{w(F)}{w(E)}p(e_{[F]})\\
&+ \sum_{\substack{F'\in S_{i}(K):\\(\exists E\in S_{i-1}(K)) F\cap F'=E}}\frac{w(F)}{w(E)}\sgn([E],\partial [F])\sgn([E],\partial [F'])p(e_{[F']}),
\end{align*}
\normalsize
and 
\small
\begin{align*}
\Delta_{i}^{down} pe_{[F]}={}&  \sum_{p(E)\in \partial p(F)} \frac{w(p(F))}{w(p(E))}e_{p([E])}\\
&+ \sum_{\substack{p(F') \in S_{i}(K):\\(\exists p(E)\in S_{i-1}(K))\\ p(F)\cap p(F')=p(E)}}\frac{w(p(F))}{w(p(E))}
\sgn(p([E]),\partial p([F]))\sgn(p([E]),\partial p([F']))e_{p([F'])},
\end{align*}
\normalsize
 it suffices to show
 \small
\begin{equation}
\label{eq:perm:part1}
\sum_{E\in \partial F} \frac{w(F)}{w(E)}=\sum_{p(E)\in \partial p(F)} \frac{w(p(F))}{w(p(E))}
\end{equation}
\normalsize
and
\small
\begin{equation}
\label{eq:perm:part2}
\frac{w(p(F))}{w(p(E))}\sgn(p([E]),\partial p([F]))\sgn(p([E]),\partial p([F']))=\frac{w(F)}{w(E)}\sgn([E],\partial [F])\sgn([E],\partial [F'])
\end{equation}
\normalsize
for every $F$ and $F'$ which are  $(i-1)$-down neighbours in $K$ and  every elementary $i$-cochain $e_{[F]}$. 
\begin{thm}
\label{circuit thm}
Let $K$ be an orientable $i$-circuit of length $m$. Then the eigenvalues of $\Delta_{i}^{down}(K)$ are $i-\cos(2\pi j/m)$, 
$j=0,1,\ldots m-1$.
\end{thm} 
 \begin{proof}
 Let  $F_{1}< F_{2}<\ldots <F_{m}$ be the ordering of  $i$-simplices of $K $ satisfying the conditions of 
Definition \ref{circuit def}. Moreover, let $[F_{1}], [F_{2}], \ldots, [F_{m}]$ be  a choice of coherent orientation on them.
Let $p:C^{i}(K,\rb)\rightarrow C^{i}(K,\rb) $ be a map  given by
 $p([F_{k}])=[F_{k+1}]$, for $1\leq k < m$ and $p([F_{m}])=[F_{1}]$. 
It is not difficult to check that
 \begin{equation}
 \label{commutativity}
 p\ldown=\ldown p
 \end{equation}
 In particular, equality (\ref{eq:perm:part1}) is satisfied since the weights of all $i$-faces are equal to $1$ and $w(F)/w(E)= w(pF)/w(pE)$.
Equality  (\ref{eq:perm:part2}) holds because $i$-faces of $K$ are coherently oriented, which gives the equalities
$\sgn([E],\partial [F])\sgn([E],\partial [F'])=-1$ and $\sgn([pE],\partial [pF])\sgn([pE],\partial [pF'])=-1,$ where $F$ and $F'$ are $(i-1)$-down neighbours of $K$ and $E$ is their intersecting face.  Hence   (\ref{commutativity}) is true.

Let $P$   be  the matrix associated to  the mapping $p$.  $P$ is a permutation matrix and  its characteristic polynomial is $\lambda^{m}-1=0$. Eigenvectors of $P$ are 
 $U_{\theta}=(1,$ $ \theta,$ $ \theta^{2},\ldots \theta^{m-1})^{T}$, where $\theta$ is the $m$-th root of unity. Thus, the eigenfunctions of the map $p$ are  $u_{\theta}([F_{k}])=\theta^{k-1}$.

With Lemma \ref{Greenman1}, we can now easily calculate  the eigenvalues of $\ldown$. \\
Let $E_{k}:=F_{k-1}\cap F_{k}$  for  $2\leq k \leq m-1$ and  let $E_{m}:=F_{m}\cap F_{1}$. We have
 \small
 \begin{align*}
 \ldown u_{\theta}([F_{k}])= {} & \sum_{\substack{E\in S_{i-1}(L):\\E\in \partial F_{k}}}\frac{w(F_{k})}{w(E)} \theta^{k-1}+ \frac{w(F_{k})}{w(E_{k})}
 \sgn([E_{k}],\partial [F_{k}])\sgn([E_{k}],\partial F_{k-1})  \theta^{k-2}\\
 & + \frac{w(F_{k})}{w(E_{k+1})}
 \sgn([E_{k+1}],\partial [F_{k}])\sgn([E_{k+1}],\partial [F_{k+1}])  \theta^{k}\\
 ={} &(\frac{2}{2}+i-1) \theta^{k-1} - \frac{1}{2}\theta^{k-2}- \frac{1}{2}\theta^{k}\\
    = {} &\theta^{k-1} (i-\frac{ \theta^{-1}+\theta}{2})\\
      ={} &\theta^{k-1}(i-\cos(\frac{2\pi j}{m})).
\end{align*}
\normalsize
It is straightforward to check that a similar equality holds for $k=1$ and $k=m$.
Thus, $\lambda_{j}=i-\cos(2\pi j/n)$, where $j=0,1,\ldots m-1$ are the eigenvalues  of $\ldown(K)$.
 \end{proof}
\begin{rem} The eigenvalues of an orientable $i$-circuit depend only on its length, thus
there are different combinatorial structures which  give  the same eigenvalues of $\ldown$. For example, $ 1, 1.5, 1.5, 2.5, 2.5, 3$ are 
the eigenvalues of 
$\Delta_{2}^{down}$ of both    simplicial complexes,  in Figure \ref{circuit center},  and the simplicial complex in Figure \ref{circuit empty}.
\end{rem}
 A similar analysis can be carried out for a non-orientable
 $i$-circuit of length $m$. In that case we define $p$ to be
 $p([F_{k}])=[F_{k+1}]$, for $1\leq k < m$ and
 $p([F_{m}])=-[F_{1}]$. The remaining calculations are carried out as
 in Theorem \ref{circuit thm}. Thus, 

\begin{thm}
Let $K $ be a non-orientable $i$-circuit of length $m$. Then the eigenvalues of $\Delta_{i}^{down}(K)$ are $i-\sin(2\pi j/m)$  for $m$ even and $i+\cos(2\pi j/m)$ for $m$ odd, where $j=0,1,\ldots m-1$.
\end{thm} 

 \begin{coll}
 Eigenvalues of $\ldown(K)$ of an  $i$-path $K$ of length $m$ are  $\lambda_{k}=i-\cos(\pi k/m)$, for $k=0,\ldots,m-1$
 \end{coll}
 \begin{proof}
 Since there are no self-intersections of dimension $(i-1)$ in an $i$-path, every path is orientable.
From  Theorem \ref{circuit thm}, we conclude that  in the spectrum of the $i$-th down Laplacian of an $i$-circuit of length $2m$, all eigenvalues appear twice, except   $(i-1)$ and $(i+1)$.
In particular, 
$\lambda_{k}=i-\cos(k\pi/m)=i-\cos((2m-k)\pi/m)=\lambda_{2m-k}$, for $k\neq 0$ and $k\neq m$.
Let $\phi=\exp(ik\pi/m)$ (where here $i=\sqrt{-1}$ should not be
confused with the same symbol $i$ for the order of the Laplace operator), then 
the eigenvector corresponding to $\lambda_{k}$ is 
$u_{k}=(1, \exp(ik\pi/m,\ldots, \exp(i(2m-1)k\pi/m)^{T}$.

The function $v_{k}=u_{k}+u_{2m-k}$ is the eigenvector for the eigenvalue $\lambda_{k}$ as well 
$$
v_{k}(m)=\e^{i\frac{\pi k}{m}}+\e^{i\frac{\pi(2m- k)}{m}}=\e^{i\frac{\pi k}{m}}+\e^{-i\frac{\pi k}{m}}.
$$

It is now a straightforward calculation to see that the first $m$
entries of $v_{k}$, for every $k=0,1,\ldots  m-1$, constitute  the
eigenvectors of $K$ for the  eigenvalue   $i-\cos(\pi k/m)$.
\end{proof}
This idea generalizes to  paths with self-intersections of dimension $(i-1)$, but then it is necessary to distinguish among orientable and non-orientable paths. The eigenvalues of a star are described in the following theorem.
\begin{thm}
Let $K $ be a  simplicial complex  consisting of $m$  $i$-simplices
assembled in a star like formation, i.e., all simplices have  one
$(i-1)$-face in common. Then the non-zero eigenvalues of $\ldown(K)$
are  
$i$ with multiplicity $(m-1)$ and $(i+1)$ with multiplicity $1$.
\end{thm}
\begin{proof}
Let $F_{k}$, $k\in \{1,\ldots,m\}$, be an $i$-dimensional face of $ K$ and let $\bigcap_{k}F_{k}=E$.
Let $p:B_{i}(K,\rb) \rightarrow B_{i}(K,\rb)$   be a permutation, such that $p([F_{k}])=[F_{k+1}]$. 
Since $F_{k}\cap F_{j}=E$, for any two $i$-faces of $K $,  we 
 can fix the orientations on the $F_{k}$ such that they induce the same orientation on  $E$. Now it is easy to check that
$$
p\ldown=\ldown p.
$$ 
Let $\theta$ denote  an $m$-th root of unity different from $1$ and $u$ the eigenvector of $p$ corresponding to it. Then we obtain
\begin{align*}
\ldown u_{\theta}([F_{k}])=& \sum_{E,E\in \partial F_{k}}\frac{w(F_{k})}{w(E)} \theta^{k-1}+ \sum_{F,F\neq F_{k}}\frac{w(F)}{w(E)}u_{\theta}([F])\\
 =& i\theta^{k-1}+\frac{1}{m}(1+\theta+\ldots+\theta^{m-1})\\
 =& i\theta^{k-1}.
\end{align*}
Thus, $u_{\theta}$ is an eigenfunction of $\ldown(K)$ corresponding to the eigenvalue $i$.
The case when $\theta=1$ results in  the eigenvalue $k+1$.
\end{proof}

\section{Regular simplcial complexes}
In this section we analyse the spectrum of the normalized Laplacian of
a regular simplicial complex, as defined in \cite{Meunier}. 
\begin{defn}
 A simplicial complex $K$ is $i$-regular iff  all its $i$-faces 
 have the same degree.
\end{defn}
Note that  a regular graph is a $0$-regular simplicial complex.
To characterize the eigenvalues of regular simplicial complexes, we introduce the notion of $i$-dual graph and $i$-path connected simplicial complexes.
\begin{defn}
\label{dual graph}
Let $K$ be a simplicial complex. Then a graph $G_{K}$ with the vertex set
$V=\{F_{j}\mid F_{j}\in S_{i}(K)\}$ and the edge set $E=\{(F_{j},F_{l})\mid F_{j}\cap F_{l}\in S_{i-1}(K)\}$ is called
an \emph{$i$-dual graph} of $K$.
Note that in graph theory dual graphs are called \emph{line graphs}.
\end{defn}
\begin{defn}
A simplicial complex $K$ is $i$-path connected  iff for any two $i$-faces $F_{1},F_{2}$ of $K$
 there exists an $i$-path connecting them.
\end{defn}
\begin{rem}
The definition of $i$-path connectedness   is different from the definition of
 $i$-connected simplicial complexes  in \cite{Kozlov}.
\end{rem}
\begin{figure}[h!tp]
  \begin{center}
    \subfigure[$2$-path connected simplicial complex and its $2$-dual graph]{\label{2connected}\includegraphics[width=4cm]{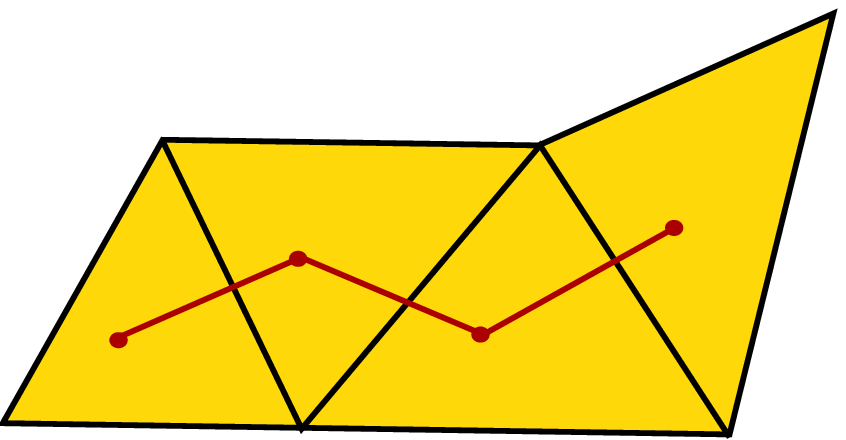}}\qquad
    \subfigure[Simplicial complex  which is not $2$-connected and its $2$-dual graph  ]{\label{2notconnected}\includegraphics[width=4cm]{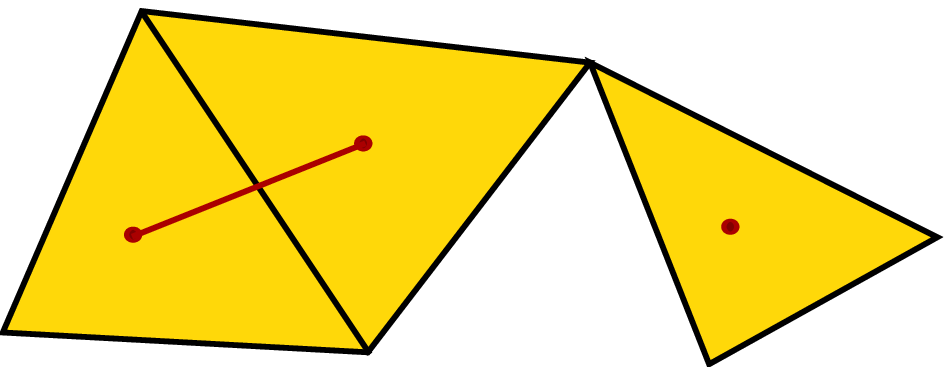}}
    \end{center}
  \caption{Examples of $i$-path connected simplicial complexes and their dual graphs}
\end{figure}
From now on, until the end of this section, we assume $K$ to be $i+1$-path connected.
\begin{thm}
\label{orientable regular sc}
 Let $K$ be an orientable $i$-regular simplicial complex, with $i$-simplices of degree $r$, and $G_{K}$ its $(i+1)$-dual graph.
Then for $r\neq 1$($r=2$)
$$ \lambda_{k}= \frac{(i+2)}{2}\mu_{k}, $$ 
where the $\lambda$'s are the eigenvalues of 
$\Delta_{i+1}^{down}$ and the $\mu$'s the eigenvalues of $\Delta_{0}^{up}(G_{K})$, both ordered non-decreasingly.
If $r=1$, then the only eigenvalue of $\Delta_{i+1}^{down}$ is  $ \lambda_{1}=(i+2)$. 
\end{thm}
\begin{proof}
Assume $r>1$.
 Since  the complex $K$ is orientable,  we can choose an orientation
 on the $(i+1)$-simplices of $K$, s.t. 
$\sgn([E],\partial [F])\sgn([E],\partial [F']=-1$, where $F$ and $F'$ are $(i+1)$-simplices and $E$ their intersecting
 face of dimension $i$. 
 Such oriented simplices uniquely determine a basis $B^{i+1}$ of $C^{i+1}$.
Hence, the matrix of the operator   $\Delta_{i+1}^{down}$ with respect
to $B^{i}$ is equal to 
$(i+2)/rI-1/r A$, where $A=(a_{ij})$ and $a_{ij}=1$ if the $(i+1)$-simplices $F_{i}$ and $F_{j}$ are $i$-down neighbours.
Assume $G_{K}$ is the $(i+1)$-dual graph of $K$, then $G_{K}$ is regular as well, and the degree of its vertices is 
$(r-1)(i+2)$. Furthermore, the adjacency matrix of $G_{K}$ equals $A$. Thus
$$\Delta_{i+1}^{down}= \frac{(2-r)(i+2)}{r}I + \frac{(r-1)(i+2)}{r}\Delta_{0}^{up}(G_{K}),$$ therefore 
$$ \lambda_{k}= \frac{(2-r)(i+2)}{r}+\frac{(r-1)(i+2)}{r}\mu_{k}. $$ 
The eigenvalue $0$ is in  $\s(\Delta_{0}^{up}(G_{K}))$, thus $(2-r)(i+2)/r$ must be in the spectrum of $\Delta_{i+1}^{down}(K)$.
Since the operator $\Delta_{i+1}^{down}$ is positive definite $(2-r)(i+2)/r\geq 0$, then $2\geq r$. Together 
with the assumption at the beginning $r>1$, we conclude that  $r$ must be equal to $2$ (another way to see that $r\leq 2$ is   from a definition of orientable simplcial complexes).
Finally,
$$ \lambda_{k}= \frac{(i+2)}{2}\mu_{k}. $$ 
If $r=1$, then $\Delta_{i+1}^{down}= (i+2)I$ and its only eigenvalue is $i+2$.
 \end{proof}
In other words, the $i$-up spectrum of the normalized Laplacian of  orientable $(i+1)$-dimensional pseudomanifolds is uniquely determined by the normalized spectrum of its dual graph.

From the previous theorem we obtain the following corollary.
\begin{coll}
Let $K$ be an $i$-regular, orientable, simplicial complex, with  eigenvalue $i+2$, then   the spectrum of $\Delta_{i+1}^{down}$ is symmetric about
$(i+2)/2$. 
\end{coll}

\begin{thm}
\label{i+2 regular sc}
 Let $K$ be an  $i$-regular simplicial complex, with $i$-simplices of degree $r$, let  $G_{K}$ be its $(i+1)$-dual graph and 
$i+2\in \s(\Delta_{i+1}^{down})$.
Then
$$ \lambda_{k}= i+2 - \frac{(r-1)(i+2)}{r}\mu_{n-k}, $$ 
where the $\lambda$'s are the eigenvalues of 
$\Delta_{i+1}^{down}$ and the $\mu$'s the eigenvalues of $\Delta_{0}^{up}(G_{K})$ both ordered non-deacreasingly, and $n$ is the number of vertices of 
$G_{K}$.
\end{thm}
\begin{proof}
 Since  $i+2\in \s(\Delta_{i+1}^{down}(K))$,  according to  Theorem \ref{condition}
 we can choose an orientation on the $(i+1)$-simplices of $K$, s.t. 
$\sgn([E],\partial [F])\sgn([E],\partial [F']=1$, for every $i$-down neighbours  $F$ and $F'$, where $F\cap F'=E$ and $\dim E=i$. 
The matrix of the operator  $\Delta_{i+1}^{down}$ is
\begin{equation}
\label{matrix equation}
\Delta_{i+1}^{down}=\frac{i+2}{r}I+\frac{1}{r}A,
\end{equation}
 where $A=(a_{ij})$, and 
\[
a_{ij}=\left\{ 
\begin{array}{ll}
1& \textrm{if } F \textrm{ and } F' \textrm{ are  } i \textrm{- down neighbours}\\
0& \textrm{otherwise.}

\end{array}
\right.
\] 
Since the degree of every vertex in the dual graph $G_{K}$ is $(r-1)(i+2)$, then 
$$\Delta_{i+1}^{down}= (i+2)I - \frac{(r-1)(i+2)}{r}\Delta_{0}^{up}(G_{K}).$$
\end{proof}
\begin{rem}
The eigenvalues of $\Delta_{i+1}^{down}$ are non-negative, hence $(i+2)-(r-1)(i+2)/r\mu_{n-k}\geq 0$, and 
\begin{equation}
\label{regular inequality 01}
 \frac{r}{r-1}\geq \mu_{n},
\end{equation}
where $\mu_{n}$ is the maximal eigenvalue of $\Delta_{0}^{up}(G_{K})$.
Inequality (\ref{regular inequality 01})  is always satisfied for  $r=2$.
\end{rem}

\section{Constructions and their effect on the spectrum: wedges,  joins and duplication of motifs }
\subsection{Wedges}
 
Let $(X_{i})_{i\in I}$ be a family of topological spaces and $x_{i}\in X_{i}$, then the wedge sum $\bigvee_{i} X_{i} $ is the quotient of their disjoint union by the identification $x_{i}\sim x_{j}$, for all $i,j\in I$, i.e.
$$
\bigvee_{i} X_{i}:= \bigsqcup_{i} X_{i}\;/ \;\{x_{i}\sim x_{j} \mid i,j \in I\}.
$$
For the purposes of this paper we define a combinatorial wedge sum, which is in many ways similar to the above  wedge sum.
\begin{defn}
\label{wedge def}
For simplicial complexes  $K_{1}$ and $K_{2}$ with vertex sets $[n]$ and  $[m]$, respectively,  and 
 $k$-simplices $F_{1}=\{v_{0},\ldots,v_{k}\}$ in $ S_{k}(K_{1})$ and $F_{2}=\{u_{0},\ldots,u_{k}\}$ in $S_{k}(K_{2})$,
 the \emph{ combinatorial $k$-wedge sum}   of $K_{1}$ and $K_{2}$ is an abstract simplicial complex on the vertex set $[m+n-k-1]$, such that
$$
K_{1}\vee_{k}K_{2}:=\{\{v_{i_{0}},\ldots,v_{i_{k}}\}\mid  \{v_{i_{0}},\ldots,v_{i_{k}}\}\in K_{1} \textrm{ or if }  \{u_{i_{0}},\ldots,u_{i_{k}}\}\in K_{2} \},
$$
where $u_{i_{j}}:=u_{l}$ if $v_{i_{j}}=v_{l} $,  $u_{i_{j}}:=v_{i_{j}}+k+1$ if $v_{i_{j}}>n$ 
and $u_{i_{j}}:=v_{i_{j}}$ for the other values of $v_{i_{j}}$.
This definition    generalizes in an obvious way  to the $k$-wedge sum of arbitrary many simplicial complexes.
\end{defn}
It is not difficult to check that $K_{1}\vee_{k}K_{2}$ is a simplicial complex, too. 
\begin{rem}
The combinatorial wedge sum $K_{1}\vee_{k}K_{2}$ can also be viewed as 
\begin{equation*}
K_{1}\sqcup K_{2}\;/ \;\{F_{1}\sim F_{2}\},
\end{equation*}
 where
$\sim$ is an equivalence relation which identifies the faces $F_{1}$ and $F_{2}$.
The combinatorial $k$-wedge sum  among graphs is a common notion in graph theory, although it is called by many different names: 
the combinatorial $0$-wedge sum of graphs is also known as vertex amalgamation \cite{ Gross2001}, coalescence \cite{GroneCoalescence} and join \cite{Jost}, whereas the combinatorial $1 $-wedge sum of graphs is called edge amalgamation.
\end{rem} 

Note that
 $K_{1}\vee_{k}K_{2}$, for arbitrary $k$, and the  
wedge sum of $K_{1}$ and $K_{2}$ as   topological spaces have isomorphic homology groups.
 From the homological point of view it is impossible to distinguish among $k$-wedge sums for different values of $k$ 
as well as among  different choices of the base points. 
 However, combinatorially, they are clearly different, see
 e.g. the 
 two wedge sums in  Figure \ref{wedgefig}.
\begin{figure}[h!tb]
\centering
\includegraphics[width=8cm]{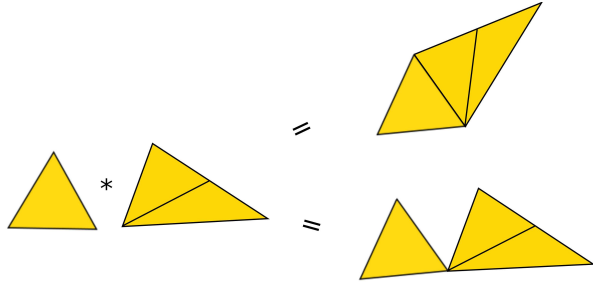}
\caption{The homology groups of the two spaces on the right are isomorphic, nonetheless these complexes are   combinatorially different.}
\label{wedgefig}
\end{figure}
Consequently, in a combinatorial $k$-wedge sum of simplicial complexes, it is important which complexes are identified as well as the dimension of these complexes.
The following theorem gives the first characterization of the effect of the wedge sum  on the spectrum of the Laplacian.
\begin{thm}
\label{thm wedges} 
$$
\s(\lup(K_{1}\vee_{k}K_{2}))\dote\s(\lup(K_{1}))\dotu \s(\lup(K_{2}))
$$
for all $i,k$ with   $0\leq k< i $.
\end{thm}
\begin{proof}

Since  $K_{1}$ and $K_{2}$ are identified by a face of dimension $k$,  then obviously, 
$C^{i}(K_{1}\vee_{k}K_{2},\rb)=C^{i}(K_{1},\rb) \oplus C^{i}(K_{2},\rb)$ for every $i>k$. Thus, the coboundary mapping
$
\delta_{i}:C^{i}(K_{1}\vee_{k}K_{2},\rb)\rightarrow C^{i+1}(K_{1}\vee_{k}K_{2},\rb)
$
will map $C^{i}(K_{j},\rb)$ to $C^{i+1}(K_{j},\rb)$, $j=1,2$ and the same holds for the adjoint  $\delta^{*}_{i}$.
\end{proof}
The operator  $\lup$ is uniquely  determined  by  the $i$ and $(i+1)$-faces  of $K$.
Hence its non-zero eigenvalues depend only on the structure of the  $(i+1)$-faces of $K$.
By abuse of notation, let $S_{i+1}(K)$ determine a pure $(i+1)$-dimensional subcomplex of $K$, whose facet set is 
$S_{i+1}(K)$. Then, there exist $k_{1},\ldots,k_{m-1}<i$, and simplicial complexes $K_{1},\ldots, K_{m}$, such that 
\begin{equation}
\label{wedgesum}
S_{i+1}(K)=K_{1}\vee_{k_{1}}K_{2}\vee_{k_{2}}\ldots\vee_{k_{m-1}}K_{m},
\end{equation}
i.e.
$$
\s(\lup(K))\dote \s(\lup(K_{1}))\dotu \ldots \dotu \s(\lup(K_{m})).
$$
Therefore, when studying $\lup$,  it is  useful to determine if $K$ can be represented as a combinatorial $k$-wedge sum 
of simplicial complexes and if so, how many of them there are.
One possible way  to answer this question  is via the   $(i+1)$-dual graph of $K$.
The number of complexes in the wedge sum (\ref{wedgesum})  is exactly
the number of connected components of the $(i+1)$-dual graph of $K$. 
It is also equal to the number of   $(i+1)$-path connected components.

\begin{rem}
If $K$ is an $(i+1)$- path connected simplicial complex, it  cannot be decomposed into a combinatorial $k$-wedge ($k<i$) of simplicial complexes.
 \end{rem}
 We collect the    above  observations in the   following proposition.
 \begin{prop}
 \label{prop wedges}
 The following statements are equivalent.
 \begin{itemize}
 \item[(i)] $S_{i+1}(K)\cong K_{1}\vee_{k_{1}}K_{2}\vee_{k_{2}}\ldots\vee_{k_{m-1}}K_{m}$, where $k_{1},\ldots,k_{m-1}<i$ and  $K_{1},\ldots, K_{m}$ are simplicial complexes.
 \item[(ii)] The $(i+1)$-dual graph $G_{K}$ of $K$ has $m$ connected components.
 \item[(iii)] The number of $(i+1)$-path connected components of  $K$ is equal to $m$.
 \end{itemize}
 \end{prop}
 
The analysis  on the combinatorial wedge sum above does not depend on the choice of the scalar products. Hence Theorem 
\ref{thm wedges} 
and Proposition \ref{prop wedges} hold for the general Laplace operator $\mathcal{L}$ as well.
In the remainder of this section we   investigate the effect  of the $k$-wedge sum for $i=k$   
on the spectrum of the (weighted)
 normalized combinatorial Laplacian $\lup$.
 
\begin{thm}
\label{join simple}
Let $K_{1}$ and $K_{2}$ be simplicial complexes, for which  the
spectra  of $\lup(K_{1})$ and $\lup(K_{2})$ both contain the eigenvalue $\lambda$,  and  let $f_{1}$, $f_{2}$ be their corresponding eigenfunctions.
If an $i$-wedge $K:=(K_{1}\vee_{i}K_{2})$ is obtained by identifying  $i$-faces $F_{1}$ and $F_{2}$, for which 
  $f_{1}([F_{1}])=f_{2}([F_{2}])$, then the spectrum of $\lup (K) $  
   contains the eigenvalue $\lambda$, too.
\end{thm}
\begin{proof}
We will prove that 
$$g([F])=\left\{  \begin{array}{ll}
f_{1}([F]) & \textrm{ for every } F \textrm{ which is an } i\textrm {-face of } K_{1} \textrm{ different from } F_{1}\\
f_{2}([F]) & \textrm{ for every } F \textrm{  which is an } i\textrm{-face of } K_{2} \\
\end{array}
\right.$$
is an eigenfunction of $\lup(K)$ corresponding to the  eigenvalue $\lambda$.
For an $i$-dimensional face $F$   of $K_{1}$ different from $F_{1}$, the following equality  holds
$$
\lup(K)\mid_{K_{1}-F_{1}}f_{1}([F])=\lambda f_{1}([F]).
$$
Similar is true when   $F\in S_{i}(K_{2})$,  $F\neq F_{2}$, i.e.
$$
\lup(K)\mid_{K_{2}-F_{2}}f_{2}([F])=\lambda f_{2}([F]).
$$
Let $w_{K_{1}}$ and $w_{K_{2}}$ denote the weight functions on the complexes $K_{1},K_{2}$ respectively.
Since we investigate $\lup$,  the weights of the $i$-simplices are
uniquely determined by the weights of the $(i+1)$-simplices and the incidence relations among them. Thus, for  the weight (degree) of the simplex $F=F_{1}=F_{2}$, we have
$w_{K}(F)=w_{K_{1}}(F_{1})+w_{K_{1}}(F_{2})$, whereas the weights of
all other simplices from $K_1$ or $K_2$ will remain the same in $K$.
Hence
\small
\begin{align*}
\lup(K)f([F])={} &\frac{1}{w_{K_{1}}(F_{1})+w_{K_{2}}(F_{2})}\sum_{\bar{F}\in S_{i+1}(K_{1}) } w_{K_{1}}(\bar{F})\sgn([F],\partial [\bar{F}])f(\partial [\bar{F}])\\
&+ \frac{1}{w_{K_{1}}(F_{1})+w_{K_{2}}(F_{2})}\sum_{\bar{F}\in S_{i+1}(K_{2}) } w_{K_{2}}(\bar{F})\sgn([F],\partial [\bar{F}])f_{2}(\partial [\bar{F}])\\
={}& \frac{w_{K_{1}}(F_{1})}{w_{K_{1}}(F_{1})+w_{K_{2}}(F_{2})} \frac{1}{w_{K_{1}}(F_{1})}\sum_{\bar{F}\in S_{i+1}(K_{1}) } w_{K_{1}}(\bar{F})\sgn([F],\partial [\bar{F}])f(\partial [\bar{F}])\\
&+ \frac{w_{K_{2}}(F_{2})}{w_{K_{1}}(F_{1})+w_{K_{2}}(F_{2})} \frac{1}{w_{K_{2}}(F_{2})}\sum_{\bar{F}\in S_{i+1}(K_{2}) } w_{K_{2}}(\bar{F})\sgn([F],\partial [\bar{F}])f_{2}(\partial [\bar{F}])\\
={} & \frac{w_{K_{1}}(F_{1})}{w_{K_{1}}(F_{1})+w_{K_{2}}(F_{2})} \lambda f_{1}([F])+ \frac{w_{K_{2}}(F_{2})}{w_{K_{1}}(F_{1})+w_{K_{2}}(F_{2})} \lambda f_{2}([F])\\
={}& \lambda f([F]).
\end{align*}
\normalsize
\end{proof}
This also includes the case when either $f_{1}$ or $f_{2}$ is identically equal to zero.
\begin{rem}
The previous  theorem will hold for the weighted normalized Laplacian  if the weight function 
$w_{K}:\bigcup_{k} S_{k}(K)\rightarrow \rb^{+}$ 
is 
\small
$$w_{K}(F)=\left\{  \begin{array}{ll}
w_{K_{1}}(F) & \textrm{ if  } F \textrm{  is a face of } K_{1} \textrm{ and } \dim F>i\\
w_{K_{2}}(F) & \textrm{ if  } F \textrm{  is a face of } K_{2} \textrm{ and } \dim F>i\\
\displaystyle\sum_{\substack{\bar{F_{1}}\in K_{1}:\\F\in \partial \bar{F_{1}}}} w_{K_{1}}(\bar{F_{1}}) + 
\sum_{\substack{\bar{F_{2}}\in K_{2}:\\F\in \partial \bar{F_{2}}}} w_{K_{2}}(\bar{F_{2}})  
& \textrm{ if  } F \textrm{  is a face of } K \textrm{ and  } \dim F\leq i\\
\end{array}
\right.$$
\normalsize
\end{rem}
\begin{exmp}
Let $\sigma_{1}$ be an $i$-simplex, then  $\s(\ldown(\sigma))\dote \s(\Delta^{up}_{i-1}(\sigma))\dote \{i+1\}$.
A function which is equal to $1$ on every oriented simplex in the boundary of
 $[\sigma]$ will be an eigenfunction of $\ldown$ corresponding to $(i+1)$.\\
According to  Theorem \ref{join simple},  an $(i-1)$-wedge of any
number of $i$-simplices will possess the  eigenvalue $(i+1)$,
 as long as we are able to orient them such that  any two  simplices whose intersection is of dimension  $i$ induce the same orientation on their intersecting face. 
For an alternative proof of this claim see Theorem \ref{i+2 eigenvalue}.
\end{exmp}
Theorem \ref{join simple}  identifies some of the eigenvalues of the   combinatorial wedge sum. 
However,  the results   obtained by using the interlacing theorem for simplicial maps, as shown in the next theorem, are more comprehensive.
\begin{thm}
Let $\mu_{1},\ldots,\mu_{m}$ be  the eigenvalues of  $\lup(K_{1}\cup
K_{2})$ and $\lambda_{1},\ldots,\lambda_{m-1}$ the eigenvalues of $\lup(K)$, where $K:=(K_{1}\vee_{i}K_{2})$, then
$$\mu_{i}\leq \lambda_{i}\leq \mu_{i+1}$$ for every $0\leq i \leq m-1$.
\end{thm}
\begin{proof}
Let $F_{1}$ and $ F_{2}$ be  $i$-faces which are identified in  an $i$-wedge sum $K$, and let  
 $f: K_{1}\cup K_{2}\rightarrow K _{1}\vee_{F_{1}\sim F_{2}}K_{2} $ be a map, which
  identifies the vertices of
 $F_{1}$ with the vertices of  $ F_{2}$, and   is the identity on the remaining vertices of $K_{1}\cup K_{2}$. Furthermore, $f$  is a simplicial map.
 The interlacing theorem  for simplicial maps (see \cite{Horak}) gives
$$
\mu_{i}\leq \lambda_{i}\leq \mu_{i+k},
$$ 
 where $k=\mid S_{i}(K_{1}\cup K_{2})\mid -\mid S_{i}(K)\mid$. 
\end{proof}
Thus the spectrum of $\lup$ of the union  of two simplicial complexes majorizes the spectrum of their $i$-wedge sum. 
\begin{rem}
 The wedge sums of graphs and its effect on the spectrum of the normalized graph
 Laplacian have already been analysed in \cite{Jost}, and   the spectrum of the combinatorial graph Laplacian was analysed in \cite{GroneCoalescence}. These are  special cases of the general theory presented here.
\end{rem}

\subsection{Joins}
Let  $K_{1}$ and $K_{2}$  be simplicial complexes on the vertex sets $[n]$ and   $[m]$, respectively.
The \emph{join} $K_{1}*K_{2}$ is a simplicial complex on the vertex set $[m+n]$,  whose faces are 
$F_{1}*F_{2}:=\{v_{0},\ldots,v_{k},n+u_{0},\ldots,n+u_{l}\}$ , where $F_{1}=\{v_{0},\ldots,v_{k}\}$ is a simplex in  $K_{1}$ and $F_{2}=\{u_{0},\ldots,u_{l}\}$ a simplex in $K_{2}$.
The cochain groups of $K_{1}*K_{2}$ are
$$
C^{i}(K_{1}*K_{2},\rb)=\bigoplus_{i_{1}+i_{2}+1=i}C^{i_{1}}(K_{1},\rb)\otimes C^{i_{2}}(K_{2},\rb),
$$
and the coboundary map $\delta_{i}$ is defined as the graded derivation
$$
\delta_{i}(f\otimes g)=\delta f\otimes g + (-1)^{\lvert f \lvert}f\otimes \delta  g,
$$
where $f\otimes g \in C^{i}(K_{1}*K_{2},\rb)$ and $\lvert f \lvert $ denotes the order of a cochain group which contains $f$.

A  natural scalar product on a tensor product of Hilbert spaces is
\begin{equation}
\label{tensorscalar}
(f_{1}\otimes g_{1},f_{2}\otimes g_{2})=(f_{1}, f_{2})_{C^{i_{1}}(K_{1})}(g_{1}, g_{2})_{C^{i_{2}}(K_{2})},
\end{equation}
where $f_{1},f_{2}\in C^{i_{1}}(K_{1})$, $g_{1},g_{2}\in C^{i_{2}}(K_{2})$.
A more general product is
\begin{equation}
\label{tensorscalar general}
(f_{1}\otimes g_{1},f_{2}\otimes g_{2})=p(i_{1})(f_{1}, f_{2})_{C^{i_{1}}(K_{1})}q(i_{2})(g_{1}, g_{2})_{C^{i_{2}}(K_{2})},
\end{equation}
where $p: \{0,1,\ldots,\dim K_{1}\}\rightarrow \rb^{+}$ and $q: \{0,1,\ldots,\dim K_{2}\}\rightarrow \rb^{+}$ are positive, real valued functions.
In terms of the weight functions,  this is
\begin{equation}
\label{tensor product weights}
w_{K_{1}*K_{2}}(F_{1}\otimes F_{2})=p(\dim F_{1}) w_{K_{1}}(F_{1}) q(\dim F_{2})w_{K_{2}}(F_{2}).
\end{equation}
An elementary calculation yields
\begin{equation}
\label{adjoint general}
\delta^{*}_{i}(f\otimes g)=\frac{p(\lvert f \lvert)}{p(\lvert f \lvert -1)}\delta^{*}f\otimes g + \frac{q(\lvert g \lvert)}{q(\lvert g \lvert -1)}(-1)^{\lvert f \lvert } f\otimes \delta^{*}g.
\end{equation}

Then the following result holds. 
\begin{thm}
\label{Duval general}

\begin{equation}
\s((\delta_{i}^{*}\delta_{i}+ \delta_{i-1}\delta_{i-1}^{*}) (K_{1}*K_{2}))\dote 
\bigcup_{\substack{\lambda_{i}\in \s((\delta_{i_{1}}^{*}\delta_{i_{1}}+ \delta_{i_{1}-1}\delta_{i_{1}-1}^{*}) (K_{1}))\\
\mu_{j}\in\s((\delta_{i_{2}}^{*}\delta_{i_{2}}+ \delta_{i_{2}-1}\delta_{i_{2}-1}^{*}) (K_{2})) }} P_{\lambda_{i}}\lambda_{i}+Q_{\mu_{j}}\mu_{j},
\end{equation}
where $i_{1}+i_{2}+1=i$, and 
$$P_{\lambda_{i}}=\left\{  \begin{array}{ll}
p(i_{1}+1)/p(i_{1}) & \textrm{ if } \lambda_{i} \in \s(\mathcal{L}_{i_{1}}^{up}(K_{1})),\\
&{}   \\
p(i_{1})/p(i_{1}-1) & \textrm{ if } \lambda_{i} \in \s(\mathcal{L}_{i_{1}}^{down}(K_{1})),
\end{array}
\right. $$ and 
$$Q_{\mu_{j}}=\left\{  \begin{array}{ll}
q(i_{2}+1)/q(i_{2}) & \textrm{ if  } \mu_{j} \in \s(\mathcal{L}_{i_{2}}^{up}(K_{2})),\\
&{}   \\
q(i_{2})/q(i_{2}-1) & \textrm{ if } \mu_{j} \in \s(\mathcal{L}_{i_{2}}^{down}(K_{2})).
\end{array}
\right. $$
\end{thm}
\begin{proof}
\begin{align}
\label{Lup general}
\delta_{i}^{*}\delta_{i}(f\otimes g)= {}&
\frac{p(\lvert f \lvert +1)}{p(\lvert f \lvert)}\delta^{*}\delta f\otimes g + (-1)^{\lvert f \lvert +1} \frac{q(\lvert g \lvert)}{q(\lvert g \lvert -1)}\delta f \otimes \delta^{*}g\\
 &+ (-1)^{\lvert f \lvert} \frac{p(\lvert f \lvert)}{p(\lvert f \lvert -1)}\delta^{*}f\otimes \delta g +  \frac{q(\lvert g \lvert +1)}{q(\lvert g \lvert)}f\otimes \delta^{*}\delta g\notag
\end{align}

\begin{align}
\label{Ldown general}
\delta_{i}\delta_{i}^{*}(f\otimes g)= {}&
\frac{p(\lvert f \lvert)}{p(\lvert f \lvert -1)}\delta\delta^{*} f\otimes g + (-1)^{\lvert f \lvert } \frac{q(\lvert g \lvert)}{q(\lvert g \lvert -1)}\delta f \otimes \delta^{*}g\\
 &+ (-1)^{\lvert f \lvert-1} \frac{p(\lvert f \lvert)}{p(\lvert f \lvert -1)}\delta^{*}f\otimes \delta g +  \frac{q(\lvert g \lvert)}{q(\lvert g \lvert -1 )}f\otimes \delta\delta^{*} g\notag
\end{align}
Addition of  (\ref{Lup general}) and (\ref{Ldown general}) gives
\begin{align}
(\delta_{i}^{*}\delta_{i}+\delta_{i}\delta_{i}^{*}) (f\otimes g)= {}&
\left(\frac{p(\lvert f \lvert +1)}{p(\lvert f \lvert)}\mathcal{L}_{\lvert f \lvert }^{up}(K_{1})+ \frac{p(\lvert f \lvert)}{p(\lvert f \lvert -1)}\mathcal{L}_{\lvert f \lvert }^{down}(K_{1})\right) f\otimes g \\
&+ f\otimes \left(\frac{q(\lvert g \lvert +1)}{q(\lvert g \lvert)}\mathcal{L}_{\lvert g \lvert }^{up}(K_{2})+ \frac{q(\lvert g \lvert)}{q(\lvert g \lvert -1)}\mathcal{L}_{\lvert g \lvert }^{down}(K_{2})\right) g\notag .
\end{align}
From the last equation we immediately deduce 
\begin{equation}
\s((\delta_{i}^{*}\delta_{i}+ \delta_{i-1}\delta_{i-1}^{*}) (K_{1}*K_{2}))\dote 
\bigcup_{\substack{\lambda_{i}\in \s(\mathcal{L}_{i_{1}} (K_{1}))\\
\mu_{j}\in\s(\mathcal{L}_{i_{2}}(K_{2})) }} P_{\lambda_{i}} \lambda_{i}+Q_{\mu_{j}}\mu_{j}.
\end{equation}
\end{proof}
\begin{rem}
Proposition 4.9. in \cite{DuvalShifted}  treats the special case of
Theorem \ref{Duval general} where the 
functions $p$ and $q$ are identically equal to $1$. In that case, the eigenvalues of these complexes satisfy
\begin{equation}
\label{eigsums}
\s((\delta_{i}^{*}\delta_{i}+ \delta_{i-1}\delta_{i-1}^{*}) (K_{1}*K_{2}))\dote 
\bigcup_{\substack{\lambda_{i}\in \s((\delta_{i_{1}}^{*}\delta_{i_{1}}+ \delta_{i_{1}-1}\delta_{i_{1}-1}^{*}) (K_{1}))\\
\mu_{j}\in\s((\delta_{i_{2}}^{*}\delta_{i_{2}}+ \delta_{i_{2}-1}\delta_{i_{2}-1}^{*}) (K_{2})) }} \lambda_{i}+\mu_{j}.
\end{equation}
In \cite{DuvalShifted}, it is  assumed that the weight functions on
the cochain spaces of $K_{1}$ and $K_{2}$ are equal to the identity,
which yields the  combinatorial Laplacian. 
\end{rem}

The next theorem provides necessary conditions on $p$ and $q$ for the
Laplace operator defined on $K_{1}*K_{2}$ to be  normalized.

\begin{thm}
\label{theorem joins}
Let  $w_{K_{1}}$ and  $w_{K_{2}}$ be the weight functions on $K_{1}$ and $K_{2}$, resp., such that      $\mathcal{L}(K_{1},w_{K_{1}})$ and 
 $\mathcal{L}(K_{2},w_{K_{2}})$ are the normalized Laplace operators.
Without  loss of generality, assume $\dim K_{1}\leq \dim K_{2}$. If
$ p(i+1)/p(i)+q(j+1)/q(j)=1$ for every $i <\dim K_{1}$ and 
 $ j< \dim K_{2}$,
then $\s(\mathcal{L}_{i}^{up}(K_{1}*K_{2},  pw_{K_{1}}qw_{K_{2}}))\subset [0, i+2]$, or in other words
$\mathcal{L}(K_{1}*K_{2}, pw_{K_{1}}qw_{K_{2}})$ is the normalized Laplacian.
\end{thm}
\begin{proof}
We check for which values of $p$ and $q$ the weight function of a join $K_{1}*K_{2}$  satisfies 
 the normalizing condition (\ref{normalizing condition}).
For arbitrary $F_{1}\in K_{1}$ and $F_{2}\in K_{2}$, we have
\small
\begin{align*}
\deg F_{1}\otimes F_{2}=&\sum_{\substack{F\in S_{i+1}(K_{1}*K_{2}):\\  F_{1}\otimes F_{2} \in \partial F}}w(F)\\
=&\sum_{\bar{F_{1}} : F_{1}\in \partial \bar{F_{1}}} w_{K_{1}*K_{2}}(\bar{F_{1}}\otimes F_{2})+
\sum_{\bar{F_{2}}: F_{2}\in \partial \bar{F_{2}}} w_{K_{1}*K_{2}}(F_{1}\otimes \bar{F_{2}})\\
=& \sum_{\bar{F_{1}} : F_{1}\in \partial \bar{F_{1}}} p(\dim \bar{F_{1}})w_{K_{1}}(\bar{F_{1}}) q(\dim F_{2})w_{K_{2}}(F_{2})\\
&+
\sum_{\bar{F_{2}}: F_{2}\in \partial \bar{F_{2}}} p(\dim F_{1}) w_{K_{1}}(F_{1}) q(\dim \bar{F_{2}})w_{K_{2}}(\bar{F_{2}})\\
=&p(\dim F_{1}+1)q(\dim F_{2})\deg F_{1} w_{K_{2}}(F_{2})+ p(\dim F_{1})q(\dim F_{2}+1)\deg F_{2} w_{K_{1}}(F_{1})\\
=& (p(\dim F_{1}+1)q(\dim F_{2})+p(\dim F_{1})q(\dim F_{2}+1) )w_{K_{1}}(F_{1})w_{K_{2}}(F_{2})
\end{align*}
\normalsize
Thus, the weight function $w_{K_{1}*K_{2}}$ satisfies (\ref{normalizing condition}) iff
$$(p(i+1)q(j)+p(i)q(j+1) )=p(i)q(j),$$ for every $i,j$. 
\end{proof}
The following corollary is a direct consequence of Theorem \ref{theorem joins} and Theorem \ref{Duval general}.

\begin{coll}
\label{Jointhm}
Let  $\dim K_{1}=d_{1}$ and $\dim K_{2}=d_{2}$ and let 
  $\mathcal{L}(K_{1},w_{K_{1}})$ and 
 $\mathcal{L}(K_{2},w_{K_{2}})$ be normalized Laplace operators. Assume $w_{K_{1}*K_{2}}:=w_{K_{1}}w_{K_{2}}$, and denote 
$\mathcal{L}(K_{1}*K_{2},w_{K_{1}*K_{2}})$ by $\Delta(K_{1}*K_{2})$, 
 then
\begin{equation}
\label{join}
\s(\Delta^{down}_{d_{1}+d_{2}+1} (K_{1}*K_{2}))\dote \bigcup_{\substack{\lambda_{i}\in \s(\Delta^{down}_{d_{1}} (K_{1}))\\
\mu_{j}\in\s(\Delta^{down}_{d_{2}} (K_{2})) }} \lambda_{i}+\mu_{j},
\end{equation}
or equivalently
$$
\s(\Delta^{up}_{d_{1}+d_{2}} (K_{1}*K_{2}))\dote \bigcup_{\substack{\lambda_{i}\in \s(\Delta^{up}_{d_{1}-1} (K_{1}))\\
\mu_{j}\in\s(\Delta^{up}_{d_{2}-1} (K_{2})) }} \lambda_{i}+\mu_{j}.
$$
\end{coll} 

\begin{proof}
For $F_{1}\in K_{1}$ and  $F_{2}\in K_{2}$ we have
\begin{align*}
\deg F_{1}\otimes F_{2}= & \sum_{\bar{F_{1}} : F_{1}\in \partial \bar{F_{1}}} w_{K_{1}}(\bar{F_{1}}) w_{K_{2}}(F_{2})+
\sum_{\bar{F_{2}}: F_{2}\in \partial \bar{F_{2}}} w_{K_{1}}(F_{1}) w_{K_{2}}(\bar{F_{2}}).\\
\end{align*}
  If neither $F_{1}$ nor $F_{2}$ is a facet of $K_{1},K_{2}$, then  the degree of $F_{1}\otimes F_{2}$ is  $ 2w_{K_{1}}(F_{1})w_{K_{2}}(F_{2})$.
Therefore, (\ref{normalizing condition})   does not hold. Consequently, the Laplace operator determined by this function 
will not be the normalized Laplace operator of  the join $K_{1}*K_{2}$.
However,  if   $F_{1}$ or $F_{2}$ is a facet, then  $\deg F_{1}\otimes F_{2}=w_{K_{1}}(F_{1})w_{K_{2}}(F_{2})$.
Thus, $w_{K_{1}*K_{2}}$ coincides with the weight function  determining $\lup(K_{1}*K_{2})$, for $i=d_{1}+d_{2}+1$.
Together with (\ref{eigsums}), this yields (\ref{join}).
\end{proof}
As a direct consequence of  Corollary \ref{Jointhm} and  the fact that $\s(\Delta_{-1}^{up}(K))= \{1\}$ we get the following corollary.
\begin{coll}
If $K$ is simplicial complex of dimension $d$, and $v*K$ a cone over $K$, then 
$$
\s(\Delta^{up}_{d} (v*K))\dote \bigcup_{\lambda_{i}\in \s(\Delta^{up}_{d-1} (K))} 1+\lambda_{i}
$$
\end{coll}
\begin{rem}[Direct product of graphs]
Direct products of graphs can be treated similarly as joins of simplicial complexes.
The direct product of two graphs $G_{1}$ and $G_{2}$ is the simplicial
complex $G$ of dimension $1$ with 
$C^{1}(G)=C^{1}(G_{1})\otimes C^{0}(G_{2})\oplus C^{0}(G_{1})\otimes C^{1}(G_{2})$ and $C^{0}(G)=C^{0}(G_{1})\otimes C^{0}(G_{2})$.
Then, by applying the same principle as in Theorem \ref{Duval general}, we obtain
$$
\s(\mathcal{L}^{up}_{0} (G_{1}\times G_{2}))\dote \bigcup_{\substack{\lambda_{i}\in \s(\mathcal{L}^{up}_{0} (G_{1}))\\
\mu_{j}\in\s(\mathcal{L}^{up}_{0} (G_{2})) }} \frac{p(1)}{p(0)}\lambda_{i}+\frac{q(1)}{q(0)}\mu_{j},
$$
where $p(0), p(1)$ and $q(0),q(1)$ are, as before, parameters of a scalar product.
This was proven by Fiedler \cite{Fiedler1973} for the  special case when $p=q\equiv 1$ and by Grigoryan in \cite{Grigoryan} for the case of the normalized graph Laplacian for $p,q$ with $p(1)/p(0)+q(1)/q(0)=1$.

Note that the extension of the direct product to higher dimensions
would lead to a cubical (instead of simplicial) complexes.
\end{rem}

\subsection{Duplication of motifs}

Let $K$ be a simplicial complex on the vertex set $[n]$ and $S$  a collection of simplices in $K$.
The \emph{closure} $\cl S$ of $S$  is the smallest subcomplex of $K$ that contains each simplex in $S$.
The \emph{star} $\st S$ of $S$  is the set of all simplices in $K$ that have a face in $S$. 
The \emph{link} $\lk S$ of $S$  is $\cl \st S - \st \cl S$.

If the subcomplex  $\Sigma$ of $K$ on the  vertices $v_{0},\ldots,v_{k}$ contains all of $K$'s faces on those vertices, then it  is called a  \emph{motif}:

\begin{defn}
\label{motifdef}
A subcomplex $\Sigma$ of a  simplicial complex $K$ is a $k$-motif iff:
\begin{itemize}
 \item [(i)] $(\forall F_{1},F_{2}\in \Sigma)\quad F_{1}, F_{2}\subset F\in K\Rightarrow F\in \Sigma$
\item [(ii)]  $\dim \lk \Sigma=k$.
  \end{itemize}
\end{defn}
In fact, as a consequence of  Theorem \ref{thm wedges}  for $i<k$  we obtain
$$\s(\lup(K))\dote \s(\lup(K-\st \Sigma)) \dotu \s(\cl\st\Sigma).$$
Therefore, it is meaningful to investigate 
 the effect of the duplication of a $k$-motif on the spectrum of $\lup$ only if $i=k$. 
\begin{rem}
If $K$ is an $(i+1)$-path connected simplicial complex, then any motif satisfying  $(i)$ in Definition \ref{motifdef}
 will have a link of dimension $i$. 
\end{rem}

Let $u_{0},\ldots,u_{m}$ be  vertices of  $\lk\Sigma$. By the definition of the link, these vertices are different from 
those in the motif $\Sigma$ ($u_{i}\neq v_{j}$, for every $0\leq i\leq m$ and $0 \leq j\leq k$). 
Let  $\Sigma'$  denote a simplicial complex on the vertices   $v_{0}',\ldots,v_{k}'$, which is isomorphic to $\Sigma$. And let 
  $f: v_{i}'\mapsto v_{i}$ be a simplicial isomorphism among these complexes.
 Then $K^{\Sigma}:=K\cup \{ \{v'_{i_{0}},\ldots,v'_{i_{l}},u_{j_{1}},\ldots,u_{j_{s}} \}\mid  \{v_{i_{0}},\ldots,v_{i_{l}},u_{j_{1}},\ldots,u_{j_{s}}\}\in K\}$.
\begin{prop}
$K^{\Sigma}$ is a simplicial complex and $\cl\st \Sigma$ is isomorphic to $\cl\st\Sigma'$.
\end{prop}
\begin{proof}
Elementary.
\end{proof}
\begin{defn}
We say that the simplicial complex $K^{\Sigma}$ is obtained from the
simplicial complex $K$ by the \emph{duplication of the $i$-motif $\Sigma$}.
\end{defn}
\begin{figure}[h!tp]
  \begin{center}
    \subfigure{\includegraphics[width=4cm]{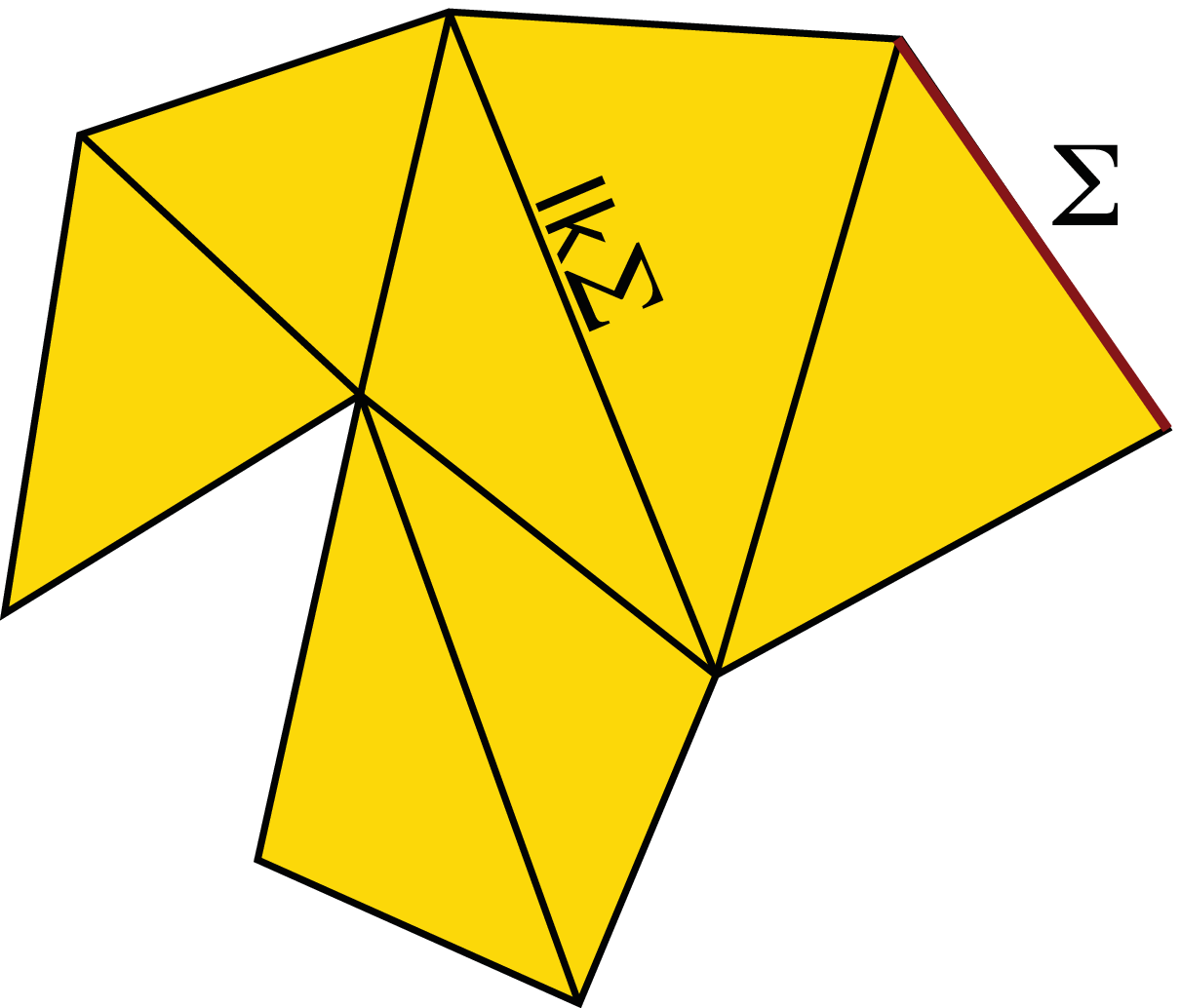}}\qquad\qquad
    \subfigure{\includegraphics[width=4cm]{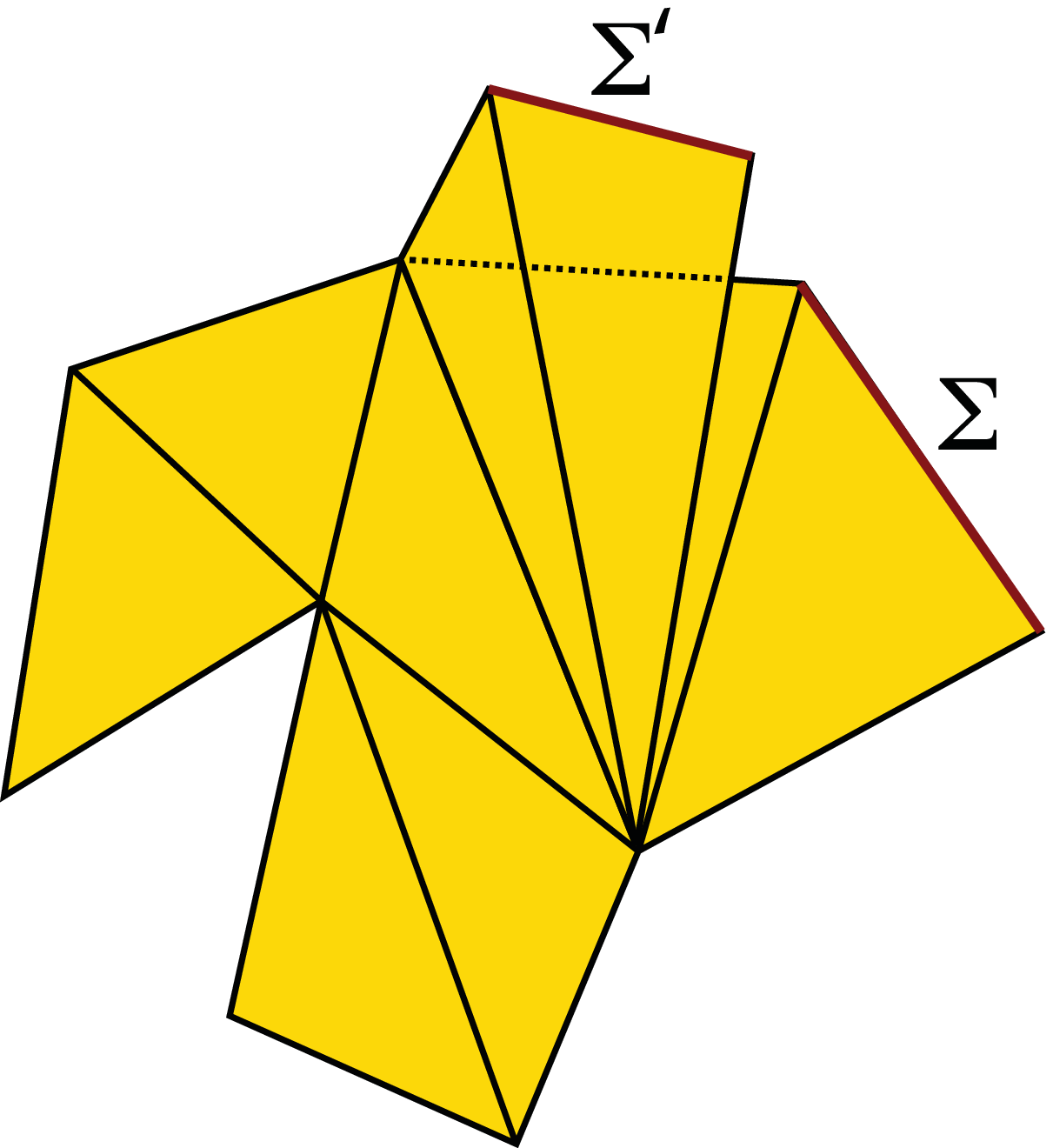}}
    \end{center}
  \caption{Duplication of  motif $\Sigma$}
\end{figure}

\begin{rem}
It could be argued  that it is $\cl\st\Sigma$ that we duplicate rather than $\Sigma$ alone.
This point of view will be very helpful in the sequel, but  we will
refer to duplication  as the duplication of the motif 
$\Sigma$, since this  is consistent with   previous work on the duplication of motifs of graphs (see \cite{Jost}).
\end{rem}
\begin{thm}
\label{main duplication} 
Let $n$ be the number of $i$-simplices in $\st \Sigma$.
Then there exist $n$ linearly independent functions $f_{1},\ldots,f_{n}$, satisfying
$$
\lup(K) f_{j}([F])=\lambda_{j}f_{j}([F]),
$$
for every $F\in S_{i}(\st\Sigma)$ and 
some  real values $\lambda_{j}$. The  doubling of the motif $\Sigma$ produces a simplicial complex $K^{\Sigma}$ with the eigenvalues $\lambda_{j}$ and  the eigenfunctions  $g_{j}$ which agree with $f_{j}$ on $\st\Sigma$ and  $-f_{j}$ on $\st\Sigma'$ and are zero elsewhere.
\end{thm}
\begin{proof}
It is trivial to check that 
$\lup(\cl \st \Sigma)$ and  $\lup(K^{ \Sigma})$  coincide on  $\st \Sigma$.
Let $\lup(\cl \st \Sigma)\mid_{ \st \Sigma} $
be the restriction of the operator $\lup(\cl \st \Sigma)$ on $\st \Sigma$.
Let $\lambda_{1},\ldots, \lambda_{n}$ be the eigenvalues  of $\lup(\cl \st \Sigma)\mid_{ \st \Sigma}$ and $f_{1},\ldots f_{n}$ the corresponding eigenfunctions.
Then 
$$g_{j}([F])=\left\{  \begin{array}{ll}
f_{j}([F]) & \textrm{ for } F \textrm{ in} \st\Sigma\\
-f_{j}([F]) & \textrm{ for} F \textrm{ in} \st\Sigma'\\
0 & \textrm{ otherwise }
\end{array}
\right.$$
 is an eigenfunction of $\lup(K^{ \Sigma})$  with  eigenvalue $\lambda_{j}$.
Without  loss of generality, assume that the labelling of the  vertices of $\Sigma$ is $v_{0},\ldots,v_{k}$ and the vertices of
$\Sigma'$ is  $v'_{0},\ldots,v'_{k}$, and they are chosen such that $v_{0}<\ldots<v_{k}<v_{0}',<\ldots<v_{k}'$.
Enumerate the vertices of $\lk  \Sigma$ with $u_{1},\ldots,u_{m}$ such that \\
  $v_{0}<\ldots<v_{k}<v_{0}',<\ldots<v_{k}'<u_{1}<\ldots <u_{m}$. Then,
   $$\lup f_{j}([F])=  \lup(\cl \st \Sigma)\mid_{ \st \Sigma} f_{j}([F])=\lambda_{j}f_{j}([F]), $$
 and 
  $$\lup(- f_{j})([F'])=  \lup(\cl \st \Sigma)\mid_{ \st \Sigma} -f_{j}([F'])=-\lambda_{j}f_{j}([F']), $$ for all $F\in S_{i}(\Sigma)$
  and $F'\in S_{i}(\Sigma')$.\\
Furthermore,  assume that $[u_{1},\ldots u_{i+1}]$ is a face of $\lk \Sigma$, then
\begin{align*}
\lup f_{j}([u_{1},\ldots u_{i+1}])={}& \sum_{v_{j}, [v_{j},u_{1},\ldots,u_{i+1}]\in S_{i+1}(\cl \st \Sigma) } (-1)^{1} f_{j}(\partial 
 [v_{j},u_{1},\ldots,u_{i+1}])\\
  &+ \sum_{v'_{j}, [v'_{j},u_{1},\ldots,u_{i+1}]\in S_{i+1}\cl \st \Sigma' } (-1)^{1}(- f_{j})(\partial 
 [v'_{j},u_{1},\ldots,u_{i+1}])\\
 ={}&0.\\
\end{align*}
Since   the functions $f_{j}$ are 0 on  the boundary of those
 $(i+1)$-simplices that are neither in $\cl \st \Sigma$ nor in $\cl \st \Sigma'$, 
 we omit them from the discussion. Hence the $\lambda_{j}$'s are the eigenvalues of $\lup(K^{\Sigma})$.
\end{proof}
As a simple consequence of Theorem \ref{main duplication} we have the following corollary.
 \begin{coll}
 \label{simple}
If the spectrum of the simplicial complex $\cl\st\Sigma$ contains the eigenvalue $\lambda$,
 with an eigenfunction $f$ that is identically equal to zero on  $\lk \Sigma$, 
then the spectrum of $K^{\Sigma}$ will contain the eigenvalue $\lambda$ as well.
\end{coll}
Theorem \ref{main duplication} is an improved and generalized version of  Theorem 2.3 from \cite{Jost}, which was stated for 
the case of the normalized graph Laplacian $\Delta_{0}^{up}$.
The  duplication of the motif $\Sigma$ will leave a specific trace in
the spectrum of the resulting simplicial complex $K^{\Sigma}$.
 In particular, if $\lambda_{1},\ldots,\lambda_{n}$ are  the
 eigenvalues of $\lup(\cl\st\Sigma)\mid_{\st\Sigma}$, then after
 duplicating  the motif $\Sigma$, $m$ times, the spectrum of the resulting complex will contain $(m-1)$ instances of every eigenvalue $\lambda_{j}$. 

It is not always straightforward to calculate the eigenvalues of $\lup(\cl\st\Sigma)\mid_{\st\Sigma}$; therefore we prove  
a theorem about interlacing of the $\lambda_{j}$ and the eigenvalues
$\mu_{j}$ of $\lup(\cl\st\Sigma)$.
With the notation of  Theorem \ref{main duplication}, we have
\begin{thm} 
The following inequality holds
$$\mu_{i}\leq \lambda_{i}\leq \mu_{i+\mid S_{i}(\lk \Sigma) \mid},$$
where $ \mid S_{i}(\lk \Sigma) \mid$ denotes the number of $i$-simplices in the link of a motif $\Sigma$.
\end{thm}
\begin{proof}
The matrix $\lup(\cl\st\Sigma)\mid_{\st\Sigma}$ is obtained from  the matrix $\lup(\cl\st\Sigma)$ by deleting 
$ \mid S_{i}(\lk \Sigma) \mid$ rows and columns. Thus, the interlacing inequality follows directly from  the
 Cauchy interlacing theorem.
\end{proof}
\begin{rem} Theorem \ref{main duplication} and  Corollary \ref{simple} will hold 
for any choice of the weight function satisfying (\ref{normalizing condition}).
\end{rem}

\section{Eigenvalues in the spectrum of $\lup$ and the combinatorial properties they encode }
One of the main advantages of the \emph{normalized} combinatorial Laplace operator is the fact that the spectrum of any simplicial complex $K$ is bounded from above by a constant.
The eigenvalues of $\lup(K)$ are in the interval $[0,i+2]$.
As this is not the case for the  spectrum of the combinatorial Laplacian $L$, or for any other known type of the
combinatorial Laplace operator $\mathcal{L}$, it  seems impossible to assign  combinatorial properties to the 
presence of a particular eigenvalue in the spectrum of  $L$ and $\mathcal{L}$.
Nonetheless,  the global properties of
the 
 spectrum  of $L_{i}$ relate to the combinatorial properties of the complex.
 For instance,  the spectrum of certain combinatorially suitable complexes is proved to be integer
 (see \cite{Dong},\cite{DuvalShifted}).

Returning to the normalized Laplacian, the appearance of the eigenvalue $2$ in the spectrum of the normalized graph Laplacian $\Delta_{0}^{up}$ means that the underlying graph is
bipartite (see \cite{Chung}),
while  the eigenvalue $1$ is produced by   duplication of motifs (see \cite{Jost}).
In the following, we characterize some of the integer eigenvalues   in the spectrum of $\lup$.
  \subsection{Eigenvalue $i+2$}
Without  loss of generality  assume  $K$ is an $(i+1)$-path connected
simplicial complex on  the vertex set $[n]$.
As shown earlier, the following inequality holds
\begin{subequations}
\begin{align}
(\lup(K) f,f)=& \sum_{\bar{F}\in S_{i+1}(K)}f(\partial [\bar{F}])^{2}w(\bar{F})\\
\label{boundary i plus 2}
\leq &(i+2)\sum_{F\in S_{i}(K)} f([F])^{2}w(F).
\end{align}
\end{subequations}
The equality in (\ref{boundary i plus 2}) is reached iff 
there exists  a function $f \in C^{i}(K,\rb)$, which satisfies 
\begin{equation*}
\sgn([F_{j}],\partial [\bar{F}])f([F_{j}])= \sgn([F_{k}],\partial [\bar{F}])f([F_{k}]),
\end{equation*}
for every $\bar{F}$ in $S_{i+1}$ and $F_{j},F_{k}\in \partial \bar{F}$.
Thus $\lvert f([F])\lvert$ must be  constant for every $F\in S_{i}(K)$.  Assume further that 
$\lvert f([F])\lvert=1$, then for every $F\in \partial \bar{F}$, $f([F])$ is equal either to $\sgn([F],\partial [\bar{F}])$ or  to 
$-\sgn([F],\partial [\bar{F}])$.
Now it is possible to consider $f$   as a choice of orientation on the
$(i+1)$-faces of $K$.

\begin{thm}
\label{condition}
 The existence of a function $f$ satisfying the equality in  (\ref{boundary i plus 2})
 is equivalent to the existence of an orientation on the
 $(i+1)$-simplices of $K$, for which any two
 $(i+1)$-simplices intersecting in a common $i$-face induce the same orientation on the intersecting 
simplex (This condition is opposite to the condition of coherently oriented simplices).
\end{thm}
\begin{thm}
 \label{i+2 eigenvalue}
For  an $i$-connected simplicial complex $K$ the following statements are equivalent
\begin{enumerate}
 \item  Spectrum $\Delta^{up}_{i}(K)$ contains the  eigenvalue $i+2$,
 \item   There are no $(i+1)$-orientable circuits of odd length nor $(i+1)$-non orientable circuits of even  length in $K$.
\end{enumerate}
\end{thm}
\begin{proof} 
$(1)\Rightarrow (2)$  
proceeds  by contradiction:
Assume that there exists an $(i+1)$-orientable circuit of odd length, whose $i$-simplices   $F_{1},\ldots,F_{2n+1}$ are
ordered increasingly, as suggested in  Definition \ref{circuit def}.
Then it is possible to orient these simplices in such a way that every two neighbouring simplices 
induce different orientations on their intersecting face. Denote these oriented simplices by $[F_{1}],\ldots,[F_{2n+1}]$.
In order to have the same orientation induced on the intersecting face, we reverse
 the orientation of every simplex $[F_{k}]$, for $k$ even. Thus, 
     $[F_{l}]$ and $-[F_{l+1}]$  induce the same orientation on $[F_{l}\cap F_{l+1}] $,   for every $1\leq l \leq 2n$. 
However, $[F_{1}]$ and $[F_{2n+1}]$   remain coherently oriented, which contradicts  Theorem ~\ref{condition}.
The analysis for the case of $(i+1)$-non-orientable circuits is analogous.\\
$(2)\Rightarrow (1)$: 
 Let $F_{1}$ be an arbitrary $(i+1)$-face of $K$. Consider its
 positive orientation $[F_{1}]$ and    call it an \emph{initial} 
oriented  face.
Let $[F_{i_{1}i_{2}\ldots i_{n}}]$ be  an $(i+1)$-face of $K$ which shares an $i$-face with $[F_{i_{1}i_{2}\ldots i_{n-1}}]$ 
and both faces induce the same orientation on their intersecting face.
Now, assume the  opposite: The eigenvalue $i+2$ is not in the spectrum
of $\lup$, i.e. it is not possible to choose an orientation on the $(i+1)$-faces of $K$, which satisfies the conditions of  Theorem \ref{condition}.
This means that after some number of steps in the construction above,  two faces
$[F_{i_{1}i_{2}\ldots i_{n}}]$, $[F_{i_{1}i_{2}\ldots i_{m}}]$  which are the same, but differently oriented are obtained.
Obviously, there exists a circuit containing  $[F_{i_{1}i_{2}\ldots i_{n}}]$, 
which does not admit an orientation as in  Theorem \ref{condition}.
This is possible only in the case when a circuit is orientable and odd
or non-orientable and even.
 This is a contradiction,
hence  $i+2$ is contained in the spectrum of $\lup$.\\
\end{proof}
The spectrum of the normalized graph Laplacian contains the eigenvalue $2$ iff the chromatic number of the underlying graph 
is $2$.
However, in general, such a connection between the chromatic number and the boundary eigenvalue in the spectrum of the normalized combinatorial Laplace operator 
only holds in one direction.
\begin{thm}
If  the chromatic number of the $1$-skeleton of the simplicial complex $K$ is $i+2$, then $i+2$ is contained
 in  $\s(\lup(K))$.
\end{thm}
\begin{proof}
Let $I_{0},\ldots,I_{i+1}$ be disjoint sets of  vertices of $K$, such that every simplex of $K$ contains at most one 
point  of each set.
Thus,  there are no  vertices  of $\bar{F}\in S_{i+1}(K)$ which are contained in the same $I_{j}$.
To avoid notational complications we relabel the vertices of $K$: 
instead of $v\in I_{j}$ ($v\in \{1,\ldots, n\}$) we write $in+v$. 
Therefore, we have 
 $$v\in I_{j}, u\in I_{k} \textrm{ and } j<k \Rightarrow v<u.$$
The function $f$, defined as
$f([v_{0},\ldots,\hat{v_{j}},\ldots,v_{i+1}])=(-1)^{j}$ ($[v_{0},\ldots,v_{i+1}],$ is 
an $(i+1)$-simplex of $K$ whose vertices are ordered increasingly, i.e. $v_{0}<\ldots<v_{i+1}$) is the eigenfunction of 
$\lup(K)$ corresponding to the eigenvalue $i+2$, i.e.
\begin{align*}
 \Delta^{up}_{i}f([F])=&\frac{\sum_{\bar{F}: F\in \partial \bar{F}}f(\partial [\bar{F}])}
{\deg F} \\
=&(i+2)f([F]).\\
\end{align*}\\
\end{proof}

\subsection{Eigenvalues $(i+1)$ and $1$}

As a special case of  Theorem \ref{main duplication} we consider  a motif $\Sigma$ consisting of only one vertex.
\begin{coll}
\label{colloraly duplication}
When we duplicate an $i$- motif $\Sigma$ consisting of one vertex which is the center
of neither an $(i+1)$-orientable odd circuit
nor an $(i+1)$-non-orientable even circuit, then we produce the eigenvalue $(i+1)$ in the spectrum of $K^{\Sigma}$.
\end{coll}
\begin{proof}
Let $v_{0}=\Sigma$  and let the $0$-simplices of $\lk \Sigma$ be $u_{1},\ldots,u_{k}$. In $\cl \st\Sigma$ all 
$(i+1)$-simplices must contain $v_{1}$. Since $v_{1}$ is neither a center of an $(i+1)$-orientable odd circuit nor
a center of an $(i+1)$-non-orientable even circuit, by  Theorem \ref{i+2 eigenvalue},  $i+2\in \mathbf{s}(\cl\st\Sigma)$.
From  Theorem ~\ref{condition}, it  follows that there is a  function $f\in C^{i}(\cl\st\Sigma,\rb)$, s.t. 
\begin{equation*}
\sgn([F_{1}],\partial [\bar{F}])f([F_{1}])=\ldots=\sgn([F_{i+2}],\partial [\bar{F}])f([F_{i+2}])
\end{equation*}
for every $\bar{F}\in S_{i+1}(\cl\st\Sigma)$ and each of its $i$-faces.
Let $g$ be a function which coincides with $f$ on oriented
 $i$-faces of $\st \Sigma$,  with $-f$ on oriented $i$-faces of $\st \Sigma'$ and is zero elsewhere.
We will now show that $g$ is an eigenfunction of $\lup(K^{\Sigma})$
associated to  the eigenvalue $(i+1)$. 
Let $F$ be an  arbitrary $i$-face of  $\st \Sigma$, then
\small
 \begin{align*}
\lup(\cl\st\Sigma)\mid_{ \st\Sigma}	 g([F])=&\frac{1}{w(F)}\sum_{\bar{F}\in S_{i+1}(\cl\st\Sigma)}
\sgn(F,\partial \bar{F})g(\partial [\bar{F}])\\
=&\frac{1}{w(F)}\sum_{\bar{F}\in S_{i+1}(\cl\st\Sigma)}\sgn([F],\partial [\bar{F}])
 \sum_{\substack{F_{j}\in \partial \bar{F}\\ F_{j}\notin \lk\Sigma}} \sgn([F_{j}],\partial [\bar{F}])f([F_{j}])\\
=&\frac{1}{w(F)}\sum_{\bar{F}\in S_{i+1}(\cl\st\Sigma)}\sgn([F],\partial [\bar{F}]) (i+1) \sgn([F],\partial [\bar{F}]) f([F])\\
=&(i+1)\frac{1}{w(F)}\sum_{\bar{F}\in S_{i+1}(\cl\st\Sigma)} f([F])\\
=&(i+1).
 \end{align*}
 \normalsize
The same analysis holds for $i$-faces of  $\st \Sigma'$.
Let   $F$  be an $i$-faces of $\cl \st \Sigma -\st \Sigma$, then  
\small
\begin{align*}
\lup(\cl\st\Sigma)\mid_{ \st\Sigma} f([F])={}&\frac{1}{w(F)}\left( \sum_{\bar{F}\in S_{i+1}(\cl\st\Sigma)}\sgn([F],\partial[\bar{F}]) g(\partial [\bar{F}])\right.\\
&+ \left. \sum_{\bar{F}\in S_{i+1}(\cl\st\Sigma')}\sgn([F],\partial [\bar{F}])g(\partial [\bar{F}]) \right) \\
={}&\frac{1}{w(F)}(i+1)\left( \sum_{\bar{F}\in S_{i+1}(\cl\st\Sigma)}g([F_{j}])+ \sum_{\bar{F}\in S_{i+1}(\cl\st\Sigma')}g([F_{j}])\right) \\
={}&\frac{1}{w(F)}(i+1)\left( \sum_{F_{j}\in S_{i}(\st\Sigma)}f([F_{j}])+ \sum_{F'_{j}\in S_{i}(\st\Sigma')}-f([F_{j}])\right) \\
 ={}& 0,
\end{align*}
\normalsize
where $F_{j}$ is a face of $\bar{F}$.
\end{proof}
This theorem is a generalization  of the vertex doubling effect on the normalized graph Laplacian $\Delta_{0}^{up}$ discussed
 in \cite{Jost}.

 In the graph case, the eigenvalue $1$ plays a very important role,
 since its multiplicity is usually significantly higher than other
 eigenvalues in graphs obtained from real world data,  see \cite{Banerjee}. 
For the Laplace operator  on higher dimensional simplicial
complexes,  the role of the eigenvalue $1$ is partially  transferred to the eigenvalue $(i+1)$ 
in higher dimensions, as shown  above.
Nevertheless, the next theorem gives a characterization  of the  eigenvalue $1$ in the spectrum of $\lup$.
\begin{thm}
Let $K$ be a simplicial complex with an eigenvalue $i+2$ in the spectrum of $\lup$ and let $G_{K}^{i}$ be its $i$-dual 
graph. Then,
$$
1\in \s(\Delta_{0}^{up}(G_{K}^{i}))\Leftrightarrow 1\in \s(\Delta_{i}^{up}(K)).
$$
\end{thm}
\begin{proof}
The multiplicity of the eigenvalue $1$ in the spectrum of $\lup(K)$ is equal to the dimension of the kernel 
of the adjacency matrix $A_{i}^{up}$ of the $i$-faces of $K$. Its entries are
$$(A_{i}^{up})_{[F],[F']}=\left\{  \begin{array}{ll}
\sgn([F],\partial [\bar{F}]) \sgn([F'],\partial [\bar{F}])& 
F \textrm{, } F' \textrm{ are $(i+1)$-up neighbours } \\
0 & \textrm{ otherwise }
\end{array}
\right.$$

Due to  Theorem \ref{condition}, 
it is possible to orient the $(i+1)$-simplices of $K$ such that   $\sgn([F],\partial [\bar{F}]) \sgn([F'],\partial [\bar{F}])$ 
is always positive.
Consequently, all entries of the matrix $A_{i}^{up}$ will be positive.
The adjacency matrices of $G_{K}^{i}$ and  $A_{i}^{up}$ are  the same, hence
the dimension of the kernel of  $A_{i}^{up}$  is equal to the
multiplicity of the eigenvalue $1$ in the
spectrum of the normalized graph Laplacian of the graph $G_{K}^{i}$.
\end{proof}
\section*{Acknowledgements}
We thank Frank Bauer and Johannes Rauh for useful suggestions. 
This work was supported by a PhD fellowship of the  International
Max-Plank Research School  "Mathematics in the Sciences" for the first
author. The second author was supported by the ERC Advanced Grant
FP7-267087 and the Volkswagen Foundation.

\bibliography{elsevier}
\bibliographystyle{plain}{}

\end{document}